\documentclass[a4paper,10pt]{article}
\usepackage[utf8x]{inputenc}
\usepackage{amsmath}
\usepackage{amssymb}
\usepackage{graphicx}
\usepackage{color}
\usepackage{colordvi}
\usepackage{graphics}
\usepackage{makeidx}
\usepackage{multicol}        
\usepackage{epsfig}
\usepackage{calc}
\newcommand{\R}{\mathbb{R}}
\newcommand{\N}{\mathbb{N}}
\newcommand{\F}{\mathbb{F}}
\newcommand{\B}{\mathbb{B}}

\newcommand{\Z}{\mathbb{Z}}

\newcommand{\cD}{\mathcal{D}}

\newcommand{\cK}{\mathcal{K}}
\newcommand{\cL}{\mathcal{L}}

\newcommand{\cW}{\mathcal{W}}
\def \one {\mathchoice {\hbox{1\kern -0.27em l}}{\hbox{1\kern -0.27em l}}
                     {\small{1\kern -0.27em l}}{\small{1\kern -0.27em l}}}
\newtheorem{proposition}{Proposition}[section]
\newtheorem{theorem}[proposition]{Theorem}
\newtheorem{corollary}[proposition]{Corollary}

\newtheorem{remark}[proposition]{Remark}
\newtheorem{example}[proposition]{Example}
\newtheorem{acknowledgement}[proposition]{Acknowledgement}

\newcommand {\proof} {\par\textit{Proof}. \ignorespaces}
\newcommand {\eproof}
      {\null\hfill{\large$\Box$}}

\begin{document}

\title{Are Quasi-Monte Carlo algorithms efficient for two-stage stochastic programs?}
\author{H. Heitsch \and H. Le\"ovey \and W. R\"omisch}
\date{\small{Humboldt-University Berlin, Institute of Mathematics, Berlin, Germany
(heitsch,leovey,romisch)@math.hu-berlin.de}}

\maketitle

\begin{abstract}
{Quasi-Monte Carlo algorithms are studied for designing discrete
approximations of two-stage linear stochastic programs. Their
integrands are piecewise linear, but neither smooth nor lie in the
function spaces considered for QMC error analysis. We show that under
some weak geometric condition on the two-stage model all terms of
their ANOVA decomposition, except the one of highest order, are 
continuously differentiable and second order mixed derivatives 
exist almost everywhere and belong to $L_{2}$. Hence, Quasi-Monte 
Carlo algorithms may achieve 
the optimal rate of convergence $O(n^{-1+\delta})$ with 
$\delta\in(0,\frac{1}{2}]$ and a constant not depending on the 
dimension if the effective dimension is close to two. The geometric 
condition is shown to be generically satisfied if the underlying 
probability distribution is normal. We discuss effective dimensions 
and dimension reduction techniques for two-stage integrands. Numerical
experiments show that indeed convergence rates close to the optimal
rate are achieved when using randomly scrambled Sobol' point sets
and randomly shifted lattice rules accompanied with suitable
dimension reduction techniques.}
\end{abstract}

\section{Introduction}
\label{intro}

Two-stage stochastic programs arise as deterministic equivalents of
improperly posed random linear programs
\begin{equation}\label{imlp}
\min\{\langle c,x\rangle:x\in X,\,Tx=h(\xi)\},
\end{equation}
where $X$ is a convex polyhedral subset of $\R^{m}$, $T$ a matrix,
$\xi$ is a $d$-dimensional random vector, $h$ represents an affine
function from $\R^{d}$ to $\R^{r}$ and $\langle\cdot,\cdot\rangle$
denotes the inner product in $\R^{m}$. The modeling idea consists in
the compensation of a possible deviation $h(\xi(\omega))-Tx$ for a
given realization $\xi(\omega)$ of $\xi$, by introducing additional
costs $\Phi(x,\xi(\omega))$ whose mean with respect to the
probability distribution $P$ of $\xi$ is added to the objective of
(\ref{imlp}). In two-stage stochastic programming it is assumed that
the additional costs represent the optimal value of a second-stage
linear program, i.e.,
\begin{equation}\label{secst}
\Phi(x,\xi)=\inf\{\langle q,y\rangle:
y\in\R^{\bar{m}},\,Wy=h(\xi)-Tx,y\ge 0\},
\end{equation}
where $W$ is a $(r,\bar{m})$-matrix called recourse matrix, 
$q\in\R^{\bar{m}}$ the recourse costs and $y$ the recourse decision.
The deterministic equivalent program then is of the form
\begin{equation}\label{twost}
\min\Big\{\langle c,x\rangle+\int_{\R^{d}}\Phi(x,\xi)P(d\xi):x\in X\Big\}.
\end{equation}
In practical applications of stochastic programming the dimension
$d$ is often large, e.g., in economics, energy, finance or
transportation (see \cite{WaZi05} for a survey of applied models).
It is worth noting that the option pricing models that served as
motivating examples for the further development of Quasi-Monte Carlo
algorithms (e.g. in \cite{WaFa03,WaSl05,WaSl11}) may be reformulated
as linear two-stage stochastic programs whose stochastic inputs are
means of geometric Brownian motions paths. So, in a sense, the
models considered here may be regarded as extensions of such
financial models (see Example \ref{option}).

The standard approach to solving the optimization model (\ref{twost}) 
consists in approximating the underlying probability distribution 
by discrete distributions $P_{n}$ based on a finite number $n$
of {\em samples} or {\em scenarios} $\xi^{j}\in\R^{d}$ with
probabilities $p_{j}$, $j=1,\ldots,n$, and to consider the 
approximate stochastic program
\[
\min\Big\{\langle c,x\rangle+\sum_{j=1}^{n}p_{j}\Phi(x,\xi^{j}):x\in X\Big\}.
\]
While the case of random
samples is studied in detail at least for independent and
identically distributed (iid) samples (see e.g. Chapters 6 and 7 in
\cite{RuSh03}, \cite[Sect. 4]{Roem03}), where the convergence rate
(in probability or quadratic mean) is $O(n^{-\frac{1}{2}})$. Only a
few papers related to stochastic programming dealt with the
situation of determi\-nistic samples with identical weights
$p_{j}=n^{-1}$ and proved (general) convergence results
(see \cite{DrHo06,PeKo05,HodM08,PfPi11}, \cite{Koiv05} for 
randomized samples or \cite{Roem10} for an overview). 

There exist two main approaches for the generation of discrete
approximations to $P$ based on deterministic samples with identical
weights. The first one is called {\em optimal quantization of
probability distributions} (see \cite{GrLu00}, \cite{Page97}) and
determines such quantizations by (approximately) solving best
approximation problems for $P$ in terms of the $L_{p}$-minimal (or
$L_{p}$-Wasserstein)  metric $\ell_{p}$, $p\ge 1$ (see Section 2.5
in \cite{RaRu98}). The primal and dual representations of $\ell_{1}$
together with a classical result (see \cite[Proposition 2.1]{Dudl69})
imply that
$$
c\,n^{-\frac{1}{d}}\leq\ell_{1}(P,P_{n})=\sup_{f\in \F_{d}\|f\|_{L}\le
1}\Big|\int_{\R^{d}}f(\xi)(P-P_{n})(d\xi)\Big|\leq\ell_{p}(P,P_{n})
$$
holds for sufficiently large $n$ and some constant $c>0$ if $P$ has
a density on $\R^{d}$ and $\F_{d}$ denotes the Banach space of Lipschitz
functions on $\R^{d}$ equipped with the Lipschitz norm $\|\cdot\|_{L}$.
This shows that the convergence rate of $\ell_{p}(P,P_{n})$ is at most
$O(n^{-\frac{1}{d}})$. This rate is indeed established in
\cite[Theorem 6.2]{GrLu00} under certain conditions on $P$. It is known
that the unit ball $\{f\in\F_{d}:\|f\|_{L}\leq 1\}$ is too large for
obtaining better rates.

The second approach utilizes {\em Quasi-Monte Carlo algorithms} that
are of the form
$$
Q_{n,d}(f)=n^{-1}\sum_{j=1}^{n}f(x^{j})\quad(n\in\N)
$$
and relies on the concept of equidistributed or low discrepancy 
point sets $\{x^{j}\}_{j=1}^{n}$ or sequences $(x^{j})_{j\in\N}$ 
in $[0,1)^{d}$ (see \cite{Sobo69,Nied92,Lemi09,DiPi10}). As observed 
in \cite{Hick98} certain reproducing kernel Hilbert spaces $\F_{d}$ 
of functions $f:[0,1]^{d}\to\R$ are particularly useful for estimating 
the quadrature error. Let $K:[0,1]^{d}\times[0,1]^{d}\to\R$ be a kernel 
satisfying $K(\cdot,y)\in\F_{d}$ and $\langle f,K(\cdot,y)\rangle=f(y)$ 
for each $y\in[0,1]^{d}$ and $f\in\F_{d}$. If $\langle\cdot,\cdot\rangle$
and $\|\cdot\|$ denote the inner product and norm in $\F_{d}$, and the 
integral
$$
I_{d}(f)=\int_{[0,1]^{d}}f(x)dx
$$
is a continuous functional on $\F_{d}$, the worst-case quadrature error
$e_{n}(\F_{d})$ allows the representation
\begin{equation}\label{qerr}
e_{n}(\F_{d})=\sup_{f\in\F_{d}\,,\|f\|\le 1}\big|I_{d}(f)-Q_{n,d}(f)\big|
=\sup_{\|f\|\le 1}|\langle f,h_{n}\rangle|=\|h_{n}\|
\end{equation}
according to Riesz' representation theorem for linear bounded functionals on
Hil\-bert spaces. The {\em representer} $h_{n}\in\F_{d}$ of the quadrature error 
is of the form
$$
h_{n}(x)=\int_{[0,1]^{d}}K(x,y)dy-n^{-1}\sum_{j=1}^{n}K(x,x^{j})\quad(\forall
x\in[0,1]^{d}).
$$
In the standard setting, the weighted tensor product Sobolev space
\cite{SlWo98}
\begin{equation}\label{wsob}
\F_{d}=\cW_{2,{\rm mix}}^{(1,\ldots,1)}([0,1]^{d})
=\bigotimes_{i=1}^{d}W_{2}^{1}([0,1])
\end{equation}
equipped with the weighted norm $\|f\|_{\gamma}^{2}=\langle
f,f\rangle_{\gamma}$ and inner product (see Section \ref{anovadec}
for the notation)
\begin{equation}\label{wnorm}
\langle f,g\rangle_{\gamma}=\sum_{u\subseteq\{1,\ldots,d\}}
\gamma_{u}^{-1}\int_{[0,1]^{|u|}}\frac{\partial^{|u|}}{\partial
x^{u}}f(x^{u},\mathbf{1}^{-u})\frac{\partial^{|u|}}{\partial
x^{u}}g(x^{u},\mathbf{1}^{-u})d x^{u},
\end{equation}
where the sequence $(\gamma_{i})$ is positive and nonincreasing, and
$\gamma_{u}$ is given by
$$
\gamma_{u}=\prod_{i\in u}\gamma_{i}
$$
for $u\subseteq\{1,\ldots,d\}$, is a reproducing kernel Hilbert space with 
the kernel
$$
K_{d,\gamma}(x,y)=\prod_{i=1}^{d}\big(1+\gamma_{i}
[1-\max\{x_{i},y_{i}\}]\big)\quad(x,y\in[0,1]^{d}).
$$
This is the so called weighted {\em anchored} tensor product Sobolev space, 
with anchor at the point $\mathbf{1}=(1,\dots,1)\in [0,1]^d$. By considering 
$\F_{d}$ now with the inner product 
$$
\langle f,g\rangle_{\gamma}=\!\!\!\sum_{u\subseteq\{1,\ldots,d\}}\!\!
\!\!\gamma_{u}^{-1}\!\!\int_{[0,1]^{|u|}}\!\!\Big(\int_{[0,1]^{d-|u|}} 
\frac{\partial^{|u|}}{\partial x^{u}}f(x)dx^{-u}\Big)\! 
\Big( \int_{[0,1]^{d-|u|}}\frac{\partial^{|u|}}{\partial
x^{u}}g(x)dx^{-u}\Big)d x^{u} 
$$
we obtain the so called weighted {\em unanchored} tensor product Sobolev 
space \cite{DSWW04,KuSS11} with the kernel
$$
K_{d,\gamma}(x,y)=\prod_{i=1}^{d}\big(1+\gamma_{i}(0.5 B_{2}(|x_{i}-y_{i}|)+
B_{1}(x_{i})B_{1}(y_{i}))\big)\quad(x,y\in[0,1]^{d}),
$$
where $B_{1}(x)=x-\frac{1}{2}$ and $B_{2}(x)=x^{2}-x+\frac{1}{6}$ are the 
Bernoulli polynomials of order $1$ and $2$, respectively.
 
Another example is a weighted  tensor productWalsh space consisting of Walsh 
series (see \cite[Example 2.8]{DiPi10} and \cite{Dick08}). These three spaces 
became important for analyzing the recently developed randomized lattice rules, 
namely, randomly shifted lattice rules \cite{SlKJ02,Kuo03,KSWW10a,NuCo06}) 
and random digitally shifted polynomial lattice rules (see \cite{Dick08,DiPi10}).
Both are special cases of {\em randomized Quasi-Monte Carlo algorithms (RQMC)}
which will be discussed in Section \ref{rqmc}. 

Here, we just mention that randomly shifted lattice rules 
\begin{equation}\label{rslr}
Q_{n,d}(\Delta,f)=n^{-1}\sum_{j=0}^{n-1}f\left(\left\{\frac{j g}{n}+\Delta\right\}\right) 
\end{equation}
can be constructed, where $\Delta$ is uniformly distributed in $[0,1)^d$, 
$g\in\Z^{d}$ is the generator of the lattice which is obtained by a 
component-by-component algorithm and $\{\cdot\}$ means taking componentwise 
the fractional part. For $f$ belonging to the weighted (un)anchored tensor 
product Sobolev space $\F_{d}$ the root mean square error of such randomly 
shifted lattice rules can be bounded by \cite{SlKJ02,Kuo03,DSWW04}
\begin{equation}\label{rate}
\sqrt{\mathbb{E}_{\Delta} \left|I_{d}(f)-Q_{n,d}(\Delta,f)\right|^{2}} 
\le C(\delta)n^{-1+\delta}\|f\|_\gamma,
\end{equation}
where the constant $C(\delta)$ does not depend on the dimension $d$
if the sequence of nonnegative weights $(\gamma_{j})$ satisfies
\begin{equation}\label{weights_decay}
\sum_{j=1}^{\infty}\gamma_{j}^{\frac{1}{2(1-\delta)}}<\infty\,.
\end{equation}

Unfortunately, typical integrands in linear two-stage stochastic
programming (see Section \ref{twostage}) do not belong to such
tensor product Sobolev or Walsh spaces and are even not of bounded
Hardy and Krause variation (on $[0,1]^{d}$). The
latter condition represents the standard requirement on the
integrand $f$ to justify Quasi-Monte Carlo algorithms via the
Koksma-Hlawka theorem \cite[Theorem 2.11]{Nied92}.

Alternatively, it is suggested in the literature to study the
so-called ANOVA decomposition (see Section \ref{anovadec}) of such
integrands, the smoothness of the ANOVA terms, effective dimensions
and/or sensitivity indices of the integrands.

The aim of the present paper is to follow the suggestions and to
derive theoretical arguments that explain why modern RQMC methods,
with focus on randomly shifted lattice rules (\ref{rslr}), converge 
with nearly the optimal rate (\ref{rate}) for the considered class 
of stochastic programs although the integrands do not satisfy 
standard requirements in QMC analysis, e.g., do not belong to 
the weighted tensor product Sobolev space (\ref{wsob}). 

As a first step in this direction we show in Section
\ref{anovatwostage} that all ANOVA terms except the one of highest
order are continuously differentiable and possess second order
partial derivatives almost everywhere under some geometric condition 
on the second stage program. In particular, the first and second order
ANOVA terms belong to the tensor product Sobolev space (\ref{wsob}).
Error estimates show that the QMC convergence rate dominates the
error if the effective superposition dimension is equal to $2$ 
(Remark \ref{r1}). In addition, we show in Section \ref{generic} 
that the geometric condition is satisfied for almost all
covariance matrices if the underlying random vector is Gaussian. The
meaning of ''almost all'' is also explained there. We also provide 
estimates of sensitivity indices and mean dimension in Section 
\ref{sens} and discuss techniques for dimension reduction. In 
accordance with the theoretical results in Section \ref{anovatwostage} 
our preliminary computational results in Section \ref{compexp} show 
that scrambled Sobol' sequences and randomly shifted lattice rules 
applied to a large scale two-stage stochastic program achieve convergence 
rates close to the optimal rate (\ref{rate}) if principal component 
analysis (PCA) is employed for dimension reduction.

\section{Randomized Quasi-Monte Carlo methods}
\label{rqmc}

Randomized Quasi-Monte Carlo algorithms (RQMC) permit us to combine 
the good features of Monte Carlo within Quasi-Monte Carlo methods 
for practical error estimation. 

If $f$ has mixed partial derivatives of second order in each variable 
in $L_2([0,1]^d)$, then the convergence rate \eqref{rate} can be 
improved to nearly $O(n^{-2})$ by embedding the function into an 
appropriate Korobov space through the so called {\em tent} or 
{\em baker's} transformation (see \cite[Section 5]{DiKS13}). 
Although this is theoretically true, this ``extra'' improved rate 
of convergence (over the already good $O(n^{-1+\delta})$) for 
smoother integrands is rarely observed for RQMC in practical 
applications of high-dimensional integration where only moderate 
or small sample sizes $n$ are affordable for computations 
\cite{Hick02}. 

A large class of QMC rules that can be randomized are the well 
known $(t,m,d)$-nets and $(t,d)$-sequences \cite{Nied92}. 
The randomization techniques for these constructions follow 
mainly two schemes: {\em random digital shifts} and {\em random
scramblings}. Random digital shifting of $(t,m,d)$-nets and 
$(t,d)$-sequences can be performed in a similar way as mentioned 
for randomly shifting lattice rules, but the operations to add 
the shift must be carried out in the basis $b$ used to define 
the $(t,m,d)$-nets (see \cite[Section 6]{DiKS13}). The 
resulting RQMC point set preserves the original net structure. 
Similar bounds for the root mean square error as in \eqref{rate} 
can be obtained for integrands belonging to the weighted 
(anchored and unanchored) tensor product Sobolev space $\F_{d}$
by using a special class of $(t,m,d)$-nets called {\em polynomial 
lattice rules}, see again \cite[Section 6]{DiKS13}. \\
The scrambling method was first introduced by Owen in \cite{Owen95}. 
The basic properties of Owen's scrambling are the following:

\begin{proposition}(\textbf{Equidistribution})\\
A randomized $(t,m,d)$-net in base $b$ using Owen's scrambling is 
again a $(t,m,d)$-net in base $b$ with probability 1. A randomized 
$(t,d)$-sequence in base $b$ using Owen's scrambling is again a 
$(t,d)$-sequence in base $b$ with probability 1.
\end{proposition}

\begin{proposition}(\textbf{Uniformity})\\
Let $\tilde{z}_i$ be the randomized version of a point $z_i$ 
originally belonging  to a  $(t,m,d)$-net in base $b$ or a 
$(t,d)$-sequence in base $b$, using Owen's scrambling. Then 
$\tilde{z}_i$ has a uniform distribution in $[0,1)^d$, 
that is, for any Lebesgue measurable set $G\subseteq[0,1)^d$, 
$P(\tilde{z}_{i}\in G)= \lambda_d(G)$, with $\lambda_d$ the 
$d$-dimensional Lebesgue measure.  
\end{proposition} 

Note that the uniformity property stated above ensures that 
the resulting RQMC estimator $\hat{Q}_{n,d}(.)$ is unbiased.
We mention here the general results about the variance of 
a RQMC estimator $\hat{Q}_{n,d}(.)$ after Owen's random 
scrambling technique to $(t,m,d)$-nets in base $b$ for 
functions $f \in L_2([0,1]^d)$ (see \cite{Owen97}).
 
\begin{theorem}\label{t1}
Let $\tilde{z}_i$, $1\le i \le n$, be the points of a 
scrambled $(t,m,d)$-net in base $b$, and let $f$ be a 
function on $[0,1)^d$ with integral $I$ and variance 
$\sigma^2:=\int (f-I_{d}(f))^2 dz  < \infty.$ 
Let $\hat{Q}_{n,d}(f)= n^{-1}\sum_{i=1}^{n}f(\tilde{z}_i)$ 
with $n=b^m$ be the RQMC estimator. Then its variance 
${\rm Var}(\hat{Q}_{n,d}(f))$ has the properties
\[ 
{\rm Var}(\hat{Q}_{n,d}(f))=o(n^{-1})\;\text{ as } n\rightarrow\infty  
\quad\text{and}\quad {\rm Var}(\hat{Q}_{n,d}(f))\leq\frac{b^t}{n}
\left(\frac{b+1}{b-1}\right)^{d}\sigma^2.
\]
For $t=0$ we have
\[ 
{\rm Var}(\hat{Q}_{n,d}(f))\leq \frac{1}{n}\left(\frac{b}{b-1} 
\right)^{d-1}\sigma^2 \le\frac{1}{n}e\sigma^2.
\]
\end{theorem}
Note that the last inequality for $t=0$ above holds since in 
this case one must have $b\ge d$. If the function $f$ has 
bounded variation in the sense of Hardy and Krause 
$V_{\rm HK}(f)< \infty$, then by the equidistribution property 
stated above the classical Koksma-Hlawka inequality holds with 
probability 1 for random scrambled $(t,m,d)$-nets, therefore 
the classical discrepancy bounds for $(t,m,d)$-nets 
\cite{Nied92} lead to 
\[
{\rm Var}(\hat{Q}_{n,d}(f))=O\left(n^{-2}(\log{n})^{2(d-1)}\right).
\]
If the integrand $f$ has a mixed partial derivatives of order 
$d$ which satisfies a H\"older condition, the above rate of 
convergence can be improved to \cite{Owen97,Owen97a}
\[
{\rm Var}(\hat{Q}_{n,d}(f))=O\left(n^{-3}(\log{n})^{d-1})\right).
\]
Further improved results for functions having {\em finite
generalized Hardy and Krause variation} can be found in
\cite[Theorem 13.25]{DiPi10}. Note, however, that
distinct from (\ref{rate}) sequences of the form 
$(n^{-\alpha}(\log{n})^{d-1})$ increase as long as
$n<\exp{\frac{d-1}{\alpha}}$ and, hence, require extremely 
large sample sizes $n$ for higher dimensons $d$ to get small.

The piecewise linear convex functions arising in stochastic 
programming (see Section \ref{twostage}) do even  not
have mixed partial derivatives (in the sense of Sobolev) 
in general. They do not have finite (generalized) Hardy and
Krause variation, too. The latter is shown for the classical
Hardy and Krause variation of the special function
$f_{d}(x)=\max\{x_1 +x_2 +\cdots + x_d -\frac{1}{2},0\}$
in \cite[Proposition 17]{Owen05},
but its proof carries over to the generalized variation. Thus, 
none of the results stated or mentioned above for RQMC can be 
used to formally justify an observed root mean square error 
convergence near to $O(n^{-1})$ (see Section \ref{compexp}) for 
integrands appearing in linear two-stage stochastic programming. 

Several modifications of the original scrambling method proposed 
by Owen have been investigated in order to provide efficient 
implementations of scramblings for practical applications, see 
the survey \cite{LELe02} and \cite{Mato98,HoHi03,TeFa03,Owen08}
for example.\\
Recent QMC constructions that aim to advantage from a setting 
with even higher smoothness of the integrands are the so called 
{\em higher order digital nets} in combination with {\em higher 
order scramblings}. For further information on this topic we 
refer the reader to \cite{Bald10,DiPi10}.

\section{Integrands of linear two-stage stochastic programs}
\label{twostage}

As described in the introduction, the integrands of two-stage linear
stochastic programs with random right-hand sides are
\begin{equation}\label{int2st}
\Phi(x,\xi)=\phi(h(\xi)-Tx),
\end{equation}
where $\phi$ denotes the optimal value function assigning to each
$t\in\R^{r}$ the infimum $\phi(t)=\inf\{\langle q,y\rangle:Wy=t,y\ge
0\}$ in $\bar{\R}=\R\cup\{-\infty,+\infty\}$. Due to duality in
linear programming, the function $\phi$ is finite and 
\begin{equation}\label{primdual}
\phi(t)=\sup\{\langle t,z\rangle:W^{\top}z\le q\},
\end{equation}
if $t\in{\rm dom}\,\phi=\{t\in\R^{r}: \phi(t)<\infty\}$ and the dual
feasible set $\cD=\{z\in\R^{r}:W^{\top}z\le q\}$ is nonempty. Here,
$q\in\R^{\bar{m}}$, $W$ is a $(r,\bar{m})$-matrix and $t$ varies in
the polyhedral cone ${\rm dom}\,\phi= W(\R_{+}^{\bar{m}})$. If $\cD$
is nonempty, it is of the form
\[
\cD={\rm conv}\{v^{1},\ldots,v^{\ell}\}+({\rm dom}\,\phi)^{*},
\]
where $v^{1},\ldots,v^{\ell}$ are the vertices of $\cD$, ${\rm
conv}$ means convex hull and $({\rm dom}\,\phi)^{*}$ is the polar
cone to the cone ${\rm dom}\,\phi=W(\R^{\bar{m}}_{+})$, i.e.,
\[
({\rm dom}\,\phi)^{*}= \{d\in\R^{r}:\langle d,t\rangle\leq 0,\forall
t\in W(\R^{\bar{m}}_{+})\}=\{d\in\R^{r}:W^{\top}d\leq 0\}.
\]
Furthermore, there exist polyhedral cones $\cK_{j}$, $j=1,\ldots,
\ell$, decomposing ${\rm dom}\,\phi$. The cone $\cK_{j}$ is the
normal cone to the vertex $v^{j}$, i.e.,
\begin{eqnarray}
\cK_{j}&=&\{t\in{\rm dom}\,\phi:\langle t,z-v^{j}\rangle\le
0,\,\forall z\in\cD\}\quad(j=1,\ldots,\ell)\\
&=&\{t\in{\rm dom}\,\phi:\langle t,v^{i}-v^{j}\rangle\le 0,\,\forall
i=1,\ldots,\ell,\,i\neq j\}. \label{nc}
\end{eqnarray}
Moreover, 
\[
\phi(t)=\langle v^{j},t\rangle\quad(\forall t\in\cK_{j})\quad
\mbox{and}\quad\phi(t)=\max_{j=1,\ldots,\ell}\langle
v^{j},t\rangle\quad(\forall t\in{\rm dom}\,\phi)
\]
and $\cup_{j=1,\ldots,\ell}\,\cK_{j}={\rm dom}\,\phi$. The
intersection $\cK_{j}\cap\cK_{j'}$ for $j\neq j'$ coincides with a
common closed face of dimension less than $d$. It is a common closed
face of dimension $d-1$ iff the two cones are adjacent. In the
latter case, the intersection is contained in
\begin{equation}\label{subspace}
\{t\in{\rm dom}\,\phi:\langle t,v^{j'}-v^{j}\rangle=0\}.
\end{equation}
If there exists $k\in\{1,\ldots,d\}$ such that the $k$th components
of $v^{j}$ and $v^{j'}$ coincide, the common closed face of
$\cK_{j}$ and $\cK_{j'}$ contains at least one of the two
one-dimensional cones
\[
\{(0,\ldots,0,t_{k},0,\ldots,0):t_{k}\ge 0\}\quad\mbox{and}\quad
\{(0,\ldots,0,t_{k},0,\ldots,0):t_{k}\le 0\}.
\]
The cones $\cK_{j}$ may also be represented by
\[
\cK_{j}=\Big\{\sum_{i\in I_{j}}\lambda_{i}w^{i}:\lambda_{i}\ge
0,\,i\in I_{j}\Big\},
\]
where $w^{i}\in\R^{r}$ are the columns of $W$ and $I_{j}=\{i\in
\{1,\ldots,\bar{m}\}:\langle w^{i},v^{j}\rangle=q_{i}\}$. Each
vertex $v^{j}$ is determined by $r$ linear independent equations out
of the $\bar{m}$ equations $\langle w^{i},v\rangle=q_{i}$,
$i=1,\ldots,\bar{m}$.

In the following we assume\\
{\bf (A1)} $h(\xi)-Tx\in W(\R_{+}^{\bar{m}})$ for all $\xi\in\R^{d}$, $x\in X$
({\em relatively complete recourse}).\\
{\bf (A2)} The dual feasible set $\cD$ is nonempty ({\em dual feasibility}).\\
{\bf (A3)} $\int_{\R^{d}}\|\xi\|P(d\xi)<\infty$ ({\em finite first moment}).\\
{\bf (A4)} $P$ has a density of the form $\rho(\xi)=
\prod_{i=1}^{d}\rho_{i}(\xi_{i})$ ($\xi\in\R^{d}$), where $\rho_{i}$
is a continuous (marginal) density on $\R$, $i=1,\ldots,d$
({\em independent components}).\\
{\bf (A5)} All components of the adjacent vertices of $\cD$ are distinct, i.e.,
all common closed faces of the normal cones to two adjacent vertices of $\cD$ 
do not parallel any coordinate axis ({\em geometric condition}).\\

Conditions (A1), (A2), (A3) imply that the two-stage stochastic
program (\ref{twost}) is well defined and represents an optimization
problem with finite convex objective and polyhedral convex feasible
set. If $X$ is compact its optimal value $v(P)$ is finite and its 
solution set $S(P)$ is nonempty, closed and convex. The quantitative 
stability results \cite[Theorems 5 and 9]{Roem03} for general 
stochastic programs imply the perturbation estimate
\begin{eqnarray}\label{valest}
|v(P)-v(Q)|&\leq&L\sup_{x\in X}\Big|\int_{\R^{d}}\Phi(x,\xi)(P-Q)(d\xi)\Big|\\
\emptyset\neq S(Q)&\subseteq& S(P)+\psi_{P}^{-1}\Big(\sup_{x\in X}\Big|
\int_{\R^{d}}\Phi(x,\xi)(P-Q)(d\xi)\Big|\Big)\B,\label{solest}
\end{eqnarray}
where $\B$ is the unit ball in $\R^{m}$, $\psi_{P}$ is the growth 
function of the objective
\[
\psi_{P}(\tau)=\inf\Big\{\langle c,x\rangle+\int_{\R^{d}}\Phi(x,\xi)
P(d\xi)-v(P):d(x,S(P))\leq\tau,\,x\in X\Big\}\quad(\tau\geq 0),
\]
its inverse is defined by $\psi_{P}^{-1}(t)=\sup\{\tau\in\R_{+}:
\psi_{P}(\tau)\leq t\}$, and $Q$ is a probability measure satisfying 
(A3), too.

For further information on linear parametric programming and
two-stage stochastic programming we refer to \cite{WaWe69,NGHB74}
and \cite{RuSh03,SDR09,Wets74}.

To give an example for (\ref{twost}) we show that option pricing 
models considered as stimulating examples for the recent developments 
in QMC theory (see e.g. \cite{WaSl05,WaSl07}) may be reformulated as 
linear two-stage stochastic programs.
\begin{example}\label{option}
Let the first stage variable $x$ represent the strike price at the
expiration date $T_{e}$. The dimensions are set to $m=1$,
$\bar{m}=2$ and the matrix $W$ is set to $W=(w,-w)$ with
$w=\exp{(rT_{e})}$ and $r$ denoting the risk-free interest rate. 
The second stage program and its dual are
\begin{eqnarray*}
\min\{y_{1}:Wy=\xi-x,y\in\R^{2},y\ge 0\}\!&=&\max\{(\xi-x)z:
z\in\R,W^{\top}z\le(1,0)^{\top}\}\\
&=&\max\{(\xi-x)z:0\le wz\le 1\}.
\end{eqnarray*}
The terminal payoff is $\exp{(-rT_{e})}\max\{0,\xi-x\}$ and $v=0$
and $v=\frac{1}{w}$ are the only vertices. Taking the expectation
then leads to the optimization model
\[
\min\Big\{-x+\int_{\R}\exp{(-rT_{e})}\max\{0,\xi-x\}\rho(\xi)d\xi:x\ge
0\Big\}.
\]
for maximizing the strike price. Now, it depends on the kind of
option how the random variable $\xi$ depends on the geometric
Brownian motion $S$ given by
\[
S_{t}=S_{0}\exp{((r-\textstyle{\frac{1}{2}}\sigma^{2})t+\sigma
W_{t})}
\]
with volatility $\sigma$ and standard Brownian motion $(W_{t})_{t\ge
0}$. For example, for arithmetic Asian options one has \cite{WaFa03}
\[
\xi=\frac{1}{d}\sum_{i=1}^{d}S_{t_{i}}\quad\mbox{with}\quad
t_{i}=\frac{iT_{e}}{d},\,i=1,\ldots,d.
\]
\end{example}
Hence, in a sense, the integrands (\ref{int2st}), (\ref{primdual}) 
extend the situations encountered in such option pricing models. 
They are, however, much more involved than in Example \ref{option}.

\section{ANOVA decomposition of integrands and effective dimension}
\label{anovadec}

The analysis of variance (ANOVA) decomposition of a function was
first proposed as a tool in statistical analysis (see \cite{Hoef48}
and the survey \cite{Take83}). In \cite{Sobo69} it was first used
for the analysis of quadrature methods.

We consider a density function $\rho$ on $\R^{d}$ and assume
(A4) from Section \ref{twostage}. As in \cite{GrKS13} we consider
the weighted $\cL_{p}$ space over $\R^{d}$, i.e.,
$\cL_{p,\rho}(\R^{d})$, with the norm
$$
\|f\|_{p,\rho}=\left\{\begin{array}{cl}\Big(\int\limits_{\R^{d}}|f(\xi)|^{p}
\rho(\xi)d\xi\Big)^{\frac{1}{p}}&\mbox{ if }1\le p<+\infty,\\
\mbox{ess}\sup\limits_{\xi\in\R^{d}}\rho(\xi)|f(\xi)|&\mbox{ if
}p=+\infty.
\end{array}\right.
$$
Let $D=\{1,\ldots,d\}$ and $f\in\cL_{1,\rho}(\R^{d})$. The
projection $P_{k}$, $k\in D$, is given by
\[
(P_{k}f)(\xi):=\int_{-\infty}^{\infty}f(\xi_{1},\ldots,\xi_{k-1},
s,\xi_{k+1},\ldots,\xi_{d})\rho_{k}(s)ds\quad(\xi\in\R^{d}).
\]
Clearly, the function $P_{k}f$ is constant with respect to
$\xi_{k}$. For $u\subseteq D$ we use $|u|$ for its cardinality, $-u$
for $D\setminus u$ and write
\[
P_{u}f=\Big(\prod_{k\in u}P_{k}\Big)(f),
\]
where the product means composition. We note that the ordering
within the product is not important because of Fubini's theorem. The
function $P_{u}f$ is constant with respect to all $\xi_{k}$, $k\in
u$. Note that $P_{u}$ satisfies the properties of a projection,
namely, $P_{u}$ is linear and $P_{u}^{2}=P_{u}$.

The ANOVA decomposition of $f\in\cL_{1,\rho}(\R^{d})$ is of the
form \cite{WaFa03,KSWW10b}
\begin{equation}\label{anovdef}
f=\sum_{u\subseteq D}f_{u}
\end{equation}
with $f_{u}$ depending only on $\xi^{u}$, i.e., on the variables
$\xi_{j}$ with indices $j\in u$. It satisfies the property
$P_{j}f_{u}=0$ for all $j\in u$ and the recurrence relation
\[
f_{\emptyset}=I_{d,\rho}(f):=P_{D}(f)\quad\mbox{and}\quad
f_{u}=P_{-u}(f)-\sum_{v\subsetneq u}f_{v}\,.
\]
It is known from \cite{KSWW10b} that the ANOVA terms are given
explicitly by
\begin{equation}\label{anova}
f_{u}=\sum_{v\subseteq u}(-1)^{|u|-|v|}P_{-v}f=P_{-u}(f)+
\sum_{v\subsetneq u}(-1)^{|u|-|v|}P_{u-v}(P_{-u}(f)),
\end{equation}
where $P_{-u}$ and $P_{u-v}$ mean integration with respect to
$\xi_{j}$, $j\in D\setminus u$ and $j\in u\setminus v$,
respectively. The second representation motivates that $f_{u}$ is
essentially as smooth as $P_{-u}(f)$ due to the Inheritance Theorem
\cite[Theorem 2]{GrKS13}. The following result is well known (e.g.
\cite{WaFa03}).
\begin{proposition}\label{p1}
If $f$ belongs to $\cL_{2,\rho}(\R^{d})$, the ANOVA functions
$\{f_{u}\}_{u\subseteq D}$ are orthogonal in
$\cL_{2,\rho}(\R^{d})$.
\end{proposition}
We define the variance of $f$ and $f_{u}$ by
$\sigma^{2}(f)=\|f-I_{d,\rho}(f)\|_{2,\rho}^{2}$, 
$\sigma_{u}^{2}(f)=\|f_{u}\|_{2,\rho}^{2}$, and have
\[
\sigma^{2}(f)=\|f\|_{2,\rho}^{2}-(I_{d,\rho}(f))^{2}=
\sum_{\emptyset\neq u\subseteq D}\|f_{u}\|^{2}_{2,\rho}
=\sum_{\emptyset\neq u\subseteq D}\sigma_{u}^{2}(f).
\]
In the literature, the ANOVA decomposition is often considered 
for functions $g\in\cL_{1}([0,1]^{d})$. Then the projections are 
defined by 
\[
(P^{\star}_{k}g)(\upsilon):=\int_{0}^{1}g(\upsilon_{1},\ldots,\upsilon_{k-1},
s,\upsilon_{k+1},\ldots,\upsilon_{d})ds\quad(\upsilon\in [0,1]^{d})
\]
and 
\[
P^{\star}_{u}g:=\Big(\prod_{k\in u}P^{\star}_{k}\Big)(g)\quad(u\subseteq D).
\]
Similarly to the case in $\R^d$ the ANOVA decomposition of 
$g\in\cL_{1}([0,1]^{d})$ is of the form 
\[
g=\sum_{u\subseteq D}g_{u}, \quad  
g_{\emptyset}:=I_{d}(g):=P^{\star}_{D}(g)\quad\mbox{and}\quad
g_{u}:=P^{\star}_{-u}(g)-\sum_{v\subsetneq u}g_{v}
\]
with $g_{u}$ depending only on $\upsilon^{u}$, i.e., on the variables
$\upsilon_{j}$ with indices $j\in u$. Note that $P^{\star}_{u}$ is 
indeed again a projection and, assuming that $g\in \cL_{2}([0,1]^{d})$, 
the same orthogonality property (now over $\cL_{2}([0,1]^{d})$) as 
in Proposition 3.1 follows.

Assuming now for simplicity that $\rho_{j}(t)>0$ for all $t\in\R$, 
$j=1,\ldots,d$, an integrand $f\in\cL_{1,\rho}(\R^{d})$ can be
transformed into a function $g$ defined on $[0,1]^d$ by inverting
the function
\begin{equation}\label{phi}
\varphi:=(\varphi_{1},\dots,\varphi_{d}),\quad\varphi_{j}(t)
:=\int_{-\infty}^{t}\rho_{j}(\xi_{j})d\xi_{j}\quad(1\le j \le d)
\end{equation}
and by defining 
\[
g(\upsilon):=
\begin{cases}
   (f\circ\varphi^{-1})(\upsilon) & \text{if } \upsilon\in (0,1)^{d}, \\
   0       & \text{if }  \upsilon\in [0,1]^{d}\setminus (0,1)^{d}.
\end{cases} 
\]
Then the ANOVA terms $g_{u}$ of $g$ are 
\begin{equation}\label{anova01}
f_{u}(\xi^{u})=g_{u}\circ\varphi_{u}(\xi^{u})\;\text{for }
\xi^{u}\in\R^{|u|},\;\;
g_{u}(\upsilon^{u})=(f_{u}\circ\varphi_{u}^{-1})(\upsilon^{u})\; 
\text{for }\upsilon^{u}\in(0,1)^{|u|},
\end{equation}
where
\[
\varphi_{u}:=(\varphi_{j_1},\ldots,\varphi_{j_{|u|}}),\;
\varphi_{u}^{-1}:=(\varphi_{j_1}^{-1},\ldots,\varphi_{j_{|u|}}^{-1}),\; 
(j_{k} \in u,\,1\le k \le |u|, \:j_k< j_l,\,k<l).
\]
When setting $\sigma^{2}_{u}(g):=\int_{[0,1]^{|u|}} g^2_{u}(\upsilon^{u})
d\upsilon^{u}$ for $\emptyset\neq u\subseteq D$ and
$\sigma^2_{\emptyset}(g):=0$ one obtains 
$\sigma^{2}_{u}(g)=\sigma^{2}_{u}(f)$ for $u\subseteq D$.

We return to the $\R^{d}$ and assume $\sigma(f)>0$ in the following
to avoid trivial cases. The {\em normalized ratios} 
$\frac{\sigma_{u}^{2}(f)}{\sigma^{2}(f)}$ serve as indicators for 
the importance of the variable $\xi^{u}$ in $f$. They are used in 
\cite{Sobo01} to define {\em global sensitivity indices}
of a set $u\subseteq D$ by
\[
S_{u}=\frac{1}{\sigma^{2}(f)}\sum_{v\subseteq
u}\sigma_{v}^{2}(f)\quad\mbox{and}\quad\bar{S}_{u}=1-S_{-u}
=\frac{1}{\sigma^{2}(f)}\sum_{v\cap u\neq\emptyset}\sigma_{v}^{2}(f).
\]
If $\bar{S}_{u}$ is small, then the variable $\xi^{u}$ is considered
inessential for $f$ in \cite{Sobo01}.

The normalized ratios are also used in \cite{Owen03,LiOw06} to
define and study the {\em dimension distribution} of a function $f$
in two ways. The dimension distribution of $f$ in the {\em
superposition (truncation) sense} is a probability measure $\nu_{S}$
($\nu_{T}$) defined on the power set of $D$ by
\[
\nu_{S}(s):=\nu_{S}(\{s\})=\!\sum_{|u|=s}\frac{\sigma_{u}^{2}(f)}{\sigma^{2}(f)}
\;\;\Big(\nu_{T}(s)=\!\!\sum_{\max\{j:j\in u\}=s}
\frac{\sigma_{u}^{2}(f)}{\sigma^{2}(f)}\Big)\;\;(s\in D).
\]
Hence, the {\em mean dimension} in the superposition (truncation)
sense is
\begin{equation}\label{meandim}
\bar{d}_{S}=\sum\limits_{\emptyset\neq u\subseteq
D}|u|\frac{\sigma_{u}^{2}(f)}{\sigma^{2}(f)}\qquad\Big(\bar{d}_{T}
=\sum_{\emptyset\neq u\subseteq D}\!\!\max\{j:j\in
u\}\frac{\sigma_{u}^{2}(f)}{\sigma^{2}(f)}\Big).
\end{equation}
It is proved in \cite[Theorem 2]{LiOw06} that the mean dimension
$\bar{d}_{S}$ in the superposition sense is closely related to the
global sensitivity indices of subsets of $D$ containing a single
element. Namely, 
\begin{equation}\label{mdimest}
\bar{d}_{S}=\sum_{j=1}^{d}\bar{S}_{\{j\}}.
\end{equation}
The paper \cite{LiOw06} also provides a formula for the dimension
variance based on $\bar{S}_{u}$ for all subsets $u$ of $D$
containing two indices.

For small $\varepsilon\in(0,1)$ ($\varepsilon=0.01$ is suggested in
a number of papers), the {\em effective superposition (truncation)
dimension} $d_{S}(\varepsilon)\in D$ ($d_{T}(\varepsilon)\in D$) is
the $(1-\varepsilon)$-quantile of $\nu_{S}$ ($\nu_{T}$), i.e.,
\begin{eqnarray*}
d_{S}(\varepsilon)&=&\min\{s\in D:\nu_{S}(u)\geq 1-\varepsilon,|u|\leq s\}\\
d_{T}(\varepsilon)&=&\min\{s\in D:\nu_{T}(\{1,\ldots,s\})\geq
1-\varepsilon\}.
\end{eqnarray*}
Note that $d_{S}(\varepsilon)\leq d_{T}(\varepsilon)$ and 
(see \cite{WaFa03,GrHo10})
\begin{equation}\label{errdim}
\max\Big\{\Big\|f-\sum_{|u|\leq
d_{S}(\varepsilon)}f_{u}\Big\|_{2,\rho},
\Big\|f-\sum_{u\subseteq\{1,\ldots,d_{T}(\varepsilon)\}}f_{u}\Big\|_{2,\rho}
\Big\}\le\sqrt{\varepsilon}\sigma(f).
\end{equation}
Small effective superposition dimension $d_{S}(\varepsilon)$, even
if $d_{T}(\varepsilon)$ is large, suggests that we may expect
superiority of QMC over MC. We note that there exist algorithms 
based on MC or QMC to compute global sensitivity indices and 
effective dimensions approximately (see
\cite{Sobo01,WaFa03,SoKu09,WaSl05} for example).
Since the algorithms are often described for functions on $[0,1]^{d}$,
we mention that
\begin{itemize}
 \item the dimension distribution and, hence, any effective dimension
 of $f$ is the same as for $g$ given by (\ref{anova01}).
 \item The algorithm of \cite{WaFa03} for estimating the effective  
 truncation dimension can be carried out equivalently for $f$, with 
 its obvious adaption to the $\R^{d}$ setting.  
\end{itemize}

All these notions are discussed in \cite{Owen03} for different classes 
of functions, including additive and multiplicative functions. We record 
here the results for additive functions for later reference.
\begin{example}\label{e0}
For functions $f$ having separability structure, i.e.,
\[
f(\xi)=\sum_{j=1}^{d}g_{j}(\xi_{j})\quad(\xi\in\R^{d})
\]
with $g_{j}\in\cL_{2,\rho_{j}}(\R)$, $j=1,\ldots,d$, the ANOVA terms
are (see \cite{Owen03})
\[
f_{\emptyset}(\xi)=\sum_{j=1}^{d}\mu_{j},\;\;f_{\{j\}}(\xi)=g_{j}(\xi_{j})-\mu_{j},\;\;
f_{u}(\xi)=0\mbox{ if }|u|>1,
\]
where $\mu_{j}=\int_{\R}g_{j}(t)\rho_{j}(t)dt$,
$\sigma_{j}^{2}=\int_{\R}(g_{j}(t)-\mu_{j})^{2}\rho_{j}(t)dt$,
$j=1,\ldots,d$. Hence, one obtains for the global sensitivity
indices, and the mean dimension in the superposition and truncation
sense, respectively,
\begin{equation}\label{sepmeandim}
S_{\{j\}}=\frac{\sigma_{j}^{2}}{\sigma^{2}(f)},\quad\bar{d}_{S}=1\quad\mbox{and}
\quad\bar{d}_{T}=\sum_{j=1}^{d}j\Big(\frac{\sigma_{j}}{\sigma(f)}\Big)^{2},
\end{equation}
while the superposition and truncation dimensions are
\[
d_{S}(\varepsilon)=1\quad(\forall\varepsilon\in(0,1))\quad\mbox{and}\quad
d_{T}(\varepsilon)=s\quad\mbox{if
}\sum_{j=s+1}^{d}\Big(\frac{\sigma_{j}}{\sigma(f)}
\Big)^{2}\leq\varepsilon
\]
with $\sigma^{2}(f)=\sum_{j=1}^{d}\sigma_{j}^{2}$.
\end{example}
The importance of the ANOVA decomposition in the context of this
paper is also due to the fact that the functions $f_{u}$ can be
(much) smoother than the original integrand $f$ under some
conditions (see \cite{GrKS10,GrKS13} and the next section).

\section{ANOVA decomposition of linear two-stage integrands}
\label{anovatwostage}

According to Section \ref{twostage} the integrands in linear
two-stage stochastic programming map from $\R^{d}$ to $\R$ and are
given by
\begin{equation}\label{integr}
f(\xi)=f_{x}(\xi)=\max_{j=1,\ldots,\ell}\langle
v^{j},h(\xi)-Tx\rangle\quad(x\in X),
\end{equation}
where the $v^{j}$, $j=1,\ldots,\ell$, are the vertices of the
dual feasible set $\cD=\{z\in\R^{r}:W^{\top}z\leq q\}$ and 
$\cK_{j}$ are the normal cones to $v^{j}$, $j=1,\ldots,\ell$.

We assume that the affine function $h$ is of the form
$h(\xi)=(\xi,\bar{h})=(\xi,0)+(0,\bar{h})$ with some fixed element
$\bar{h}\in\R^{r-d}$. The integrands are parametrized by the
first-stage decision $x$ varying in $X$. Such functions do not 
belong to the tensor product Sobolev spaces described in Section
\ref{intro} and, in general, are not of bounded variation in the 
sense of Hardy and Krause (see \cite[Proposition 17]{Owen05}).

Next we intend to compute projections $P_{k}(f)$ for $k\in D$. Let $x\in X$
be fixed, $\xi_{i}\in\R$, $i=1,\ldots,d$, $i\neq k$, be given. We set
$\xi^{k}=(\xi_{1},\ldots,\xi_{k-1},\xi_{k+1},\ldots,\xi_{d})$
and $\xi_{s}^{k}=(\xi_{1},\ldots,\xi_{k-1},s,\xi_{k+1},\ldots,\xi_{d})$. 
We assume (A1)--(A5) and have according to Section \ref{twostage}
\[
(\xi_{s}^{k},\bar{h})-Tx\in{\rm
dom}\,\phi=\bigcup_{j=1}^{\ell}\,\cK_{j}
\]
for every $s\in\R$ and by definition of the projection
\begin{equation}\label{pkf}
(P_{k}f(\xi^{k})=\int_{-\infty}^{\infty}f(\xi_{s}^{k})\rho_{k}(s)ds
=\int_{-\infty}^{\infty}f(\xi_{1},\ldots,\xi_{k-1},s,\xi_{k+1},\ldots,\xi_{d})
\rho_{k}(s)ds.
\end{equation}
The one-dimensional affine subspace $\{(\xi_{s}^{k},\bar{h})-Tx: s\in\R\}$ 
intersects a finite number of the polyhedral cones
$\cK_{j}$. Hence, there exist $p=p(k)\in\N\cup\{0\}$, $s_{i}=s_{i}^{k}\in\R$, 
$i=1,\ldots,p$, and $j_{i}=j_{i}^{k}\in\{1,\ldots,\ell\}$,
$i=1,\ldots,p+1$, such that $s_{i}<s_{i+1}$ and
\begin{eqnarray*}
(\xi_{s}^{k},\bar{h})-Tx&\in&\cK_{j_{1}}\quad\quad\forall s\in(-\infty,s_{1}]\\
(\xi_{s}^{k},\bar{h})-Tx&\in&\cK_{j_{i}}\quad\quad\forall s\in[s_{i-1},s_{i}]\quad(i=2,\ldots,p)\\
(\xi_{s}^{k},\bar{h})-Tx&\in&\cK_{j_{p+1}}\quad\forall
s\in[s_{p},+\infty).
\end{eqnarray*}
By setting $s_{0}:=-\infty$, $s_{p+1}:=\infty$, we obtain the
following explicit representation of $P_{k}f$
\begin{equation}\label{pkf1}
(P_{k}f)(\xi^{k})=\sum_{i=1}^{p+1}\int_{s_{i-1}}^{s_{i}}
\langle v^{j_{i}},(\xi_{s}^{k},\bar{h})-Tx\rangle\rho_{k}(s)ds,
\end{equation}
where the points $s_{i}$, $i=1,\ldots,p$, satisfy the equations
\begin{eqnarray*}
0&=&\langle(\xi_{s_{i}}^{k},\bar{h})-Tx,v^{j_{i+1}}-v^{j_{i}}\rangle
=\langle(\xi_{s_{i}}^{k},0)+(0,\bar{h})-Tx,v^{j_{i+1}}-v^{j_{i}}\rangle\\
&=&\sum_{j=1\atop{j\neq k}}^{d}\xi_{j}(v_{j}^{j_{i+1}}-v_{j}^{j_{i}})
+s_{i}(v_{k}^{j_{i+1}}-v_{k}^{j_{i}})+\langle(0,\bar{h})-Tx,v^{j_{i+1}}-v^{j_{i}}\rangle
\end{eqnarray*}
according to (\ref{subspace}). By setting $w^{i}=v^{j_{i+1}}-v^{j_{i}}$ for
$i=1,\ldots,p$ and $z(x)=(0,\bar{h})-Tx$ this leads to the explicit formula
\begin{equation}\label{si}
s_{i}=s_{i}(\xi^{k},x)=\frac{1}{w_{k}^{i}}\Big[-\sum_{j=1\atop{j\neq k}}^{d}
w_{j}^{i}\xi_{j}-\langle z(x),w^{i}\rangle\Big]\quad(i=1,\ldots,p).
\end{equation}
Hence,  all $s_{i}$, $i=1,\ldots,p$, are affine functions of the remaining 
components $\xi_{j}$, $j\neq k$. The first step in our analysis consists in 
studying smoothness properties of the projection $P_{k}f$ on $\R^{d}$. 
We note that $f$ and $P_{k}f$ are finite convex functions on $\R^{d}$ and,
hence, twice differentiable almost everywhere due to Alexandroff's theorem 
(see, for example, \cite[Section 6.4]{EvGa92}). Our analysis shows that the 
integration in (\ref{pkf}) improves the smoothness properties.  

In the following, we consider a point $\xi^{k}_0 \in \mathbb{R}^{d-1}$ and an
open ball $\B_{\epsilon_0}(\xi^{k}_0)$. Assume that the ball 
$\B_{\epsilon_0}(\xi^{k}_0)$ is small enough such that the set of cones 
\[
\cK_{\epsilon}(\xi_{0}^{k}):=\left\{ \cK_{j}: \: \: \cK_{j}\cap 
\{(\xi_{s}^{k},\bar{h})-Tx: s\in\R\}\neq\emptyset\text{ for some }\xi^{k}\in\B_{\epsilon}(\xi^{k}_0)\right\} 
\]
satisfies $\cK_{\epsilon}(\xi_{0}^{k})=\cK_{\epsilon_0}(\xi_{0}^{k})$ for $0<\epsilon<\epsilon_0$. 
Thus, the relevant interception cones are fixed in a neighboring ball 
$\B_{\epsilon}(\xi^{k}_0)$ of $\xi^{k}_0$. 
We consider also the sets of intercepted cones at an arbitrary point $\xi_{}^{k} \in B_{\epsilon}(\xi^{k}_0)$ 
\[
\cK(\xi_{}^{k}):=\left\{ \cK_{j}: \: \: \cK_{j}\cap \{(\xi_{s}^{k},\bar{h})-
Tx: s\in\R\} \neq \emptyset \right\}. 
\] 
Note that each polyhedral convex cone $\cK_j$ in $\cK_{\epsilon_0}(\xi_{0}^{k})$ contains at least one point of the affine one dimensional space 
$\{(\xi_{0_s}^{k},\bar{h})-Tx: s\in\R\}$, therefore we have 
\[
\cK(\xi_{0}^{k})=\cK_{\epsilon_0}(\xi_{0}^{k}). 
\]
Moreover, since the cones $\cK_j$ are convex, the intersection of a cone 
$\cK_j$ with the affine one dimensional space
$\{(\xi_{0,s}^{k},\bar{h})-Tx: s\in\R\}$ is given either by a single point or 
by an interval. In case that the intersection is given by an interval 
$I_{\cK_j}(\xi_{0}^{k})$, we have due to (A5) that the interior of 
$I_{\cK_j}(\xi_{0}^{k})$, denoted $I^{\circ}_{\cK_j}(\xi_{0}^{k})$, contains 
only interior points of $\cK_j$ and, hence, of $\R^{d}$. This is true because 
otherwise the interval $I_{\cK_j}(\xi_{0}^{k})$ must lie in a facet of $\cK_j$, 
and this would imply that there 
is a facet that is parallel to one of the canonical basis elements $e_{k}$, 
$1\le k \le d$, in contradiction to (A5). This implies that we can partition 
the affine one-dimensional space $\{(\xi_{0,s}^{k},\bar{h})-Tx: s\in\R\}$  
by considering the intervals $(s_i,s_{i+1})$, $0\le i \le p(\xi_{0}^{k})+1$, 
such that 
\[
\{(\xi_{0,s}^{k},\bar{h})-Tx: s \in (s_i,s_{i+1}) \} \subset \cK^{\circ}_{j_i}. 
\]
Recall that $s_0=-\infty$ and $s_{p(\xi_{0}^{k})+1}=+\infty$. It follows also 
that for each point $s_i, \, 1\le i \le p(\xi_{0}^{k})$, the resulting point 
$\{(\xi_{0,{s_i}}^{k},\bar{h})-Tx\} \in \R^{r}$  satisfies 
\[
\{(\xi_{0_{s_i}}^{k},\bar{h})-Tx\}=\bigcap_{j: \cK_j \in \varLambda_{s_i} } \cK_j,
\]
for a set of cones $\varLambda_{s_i}\subset \cK(\xi_{0}^{k}) $, and that we have 
$\varLambda_{s_r}\cap \varLambda_{s_{r+1}}=\cK_{j_{r+1}}$, $\varLambda_{s_r}\cap 
\varLambda_{s_{r+n}}=\emptyset$ for $n\ge2$, and 
$\cK(\xi_{0}^{k})=\bigcup_{\, 1\le i \le p(\xi_{0}^{k})}\varLambda_{s_i}$.

Now, we are ready to state our first result on smoothness properties of $P_{k}f$.

\begin{theorem}\label{t2}
Let $k\in D$ and $x\in X$. Assume (A1)--(A5) and let $f=f_{x}$ be the 
integrand (\ref{integr}) of the linear two-stage stochastic program (\ref{twost}).
Then the $k$th projection $P_{k}f$ of $f$ is continuously differentiable on 
$\R^{d}$. $P_{k}f$ is $s$-times continuously differentiable almost everywhere 
if the density $\rho_{k}$ belongs to $C^{s-2}(\R)$ for some $s\in\N$, $s\ge 2$. 
\end{theorem}

\begin{proof} 
In the following, we consider two possible cases for a given ball 
$\B_{\epsilon_0}(\xi^{k}_0)$ satisfying the requirement described above:\\
{\bf $P1$.)} The set of intercepted cones is the same for every 
$\xi_{}^{k}\in \B_{\epsilon_0}(\xi^{k}_0)$, that is, 
\[
\cK(\xi^{k})=\cK(\xi_{0}^{k})\;\mbox{ for all }\xi^{k}\in\B_{\epsilon_0}(\xi^{k}_0). 
\]   
{\bf $P2$.)} The set of intercepted cones $\cK(\xi_{}^{k})$ varies for  
$\xi^{k} \in \B_{\epsilon_0}(\xi^{k}_0)$.\\
 
For the case {\bf $P1$.)}, we have that the limiting functions $s_{i}$ are 
differentiable over the entire neighborhood $\B_{\epsilon_0}(\xi^{k}_0)$, 
because they admit a representation as an affine function over the whole ball 
$\B_{\epsilon_0}(\xi^{k}_0)$. Thus, we obtain from (\ref{pkf1}) for any $l\in D$, 
$l\neq k$, that $P_{k}f$ is partially differentiable with respect to $\xi_{l}$ 
at $\xi^{k}$ and
\begin{eqnarray*}
\frac{\partial P_{k}f}{\partial\xi_{l}}(\xi^{k})&=&\sum_{i=1}^{p+1}
\frac{\partial}{\partial\xi_{l}}\int_{s_{i-1}}^{s_{i}}
\langle v^{j_{i}},(\xi_{s}^{k},\bar{h})-Tx\rangle\rho_{k}(s)ds\\
&=&\sum_{i=1}^{p+1}\int_{s_{i-1}}^{s_{i}}v^{j_{i}}_{l}\rho_{k}(s)ds 
+\sum_{i=1}^{p}\langle v^{j_{i}},\xi_{s_{i}}^{k},\bar{h})-Tx\rangle
\rho_{k}(s_{i})\frac{\partial s_{i}}{\partial\xi_{l}} \\
&&-\sum_{i=2}^{p+1}\langle v^{j_{i}},(\xi_{s_{i-1}}^{k},\bar{h})-Tx\rangle
\rho_{k}(s_{i-1})\frac{\partial s_{i-1}}{\partial\xi_{l}}\\
&=&\sum_{i=1}^{p+1}v^{j_{i}}_{l}\int_{s_{i-1}}^{s_{i}}\rho_{k}(s)ds 
=\sum_{i=1}^{p+1}v^{j_{i}}_{l}(\varphi_{k}(s_{i})-\varphi_{k}(s_{i-1})),
\end{eqnarray*}
where we used the identity 
$\langle v^{j_{i}},(\xi_{s_{i}}^{k},\bar{h})-Tx\rangle=
\langle v^{j_{i+1}},(\xi_{s_{i}}^{k},\bar{h})-Tx\rangle$ for each $i=1,\ldots,p$
and $\varphi_{k}$ denotes the marginal distribution function with density
$\rho_{k}$. By reordering the latter sum we have
\begin{equation}\label{firstpartial}
\frac{\partial P_{k}f}{\partial\xi_{l}}(\xi^{k})=-\sum_{i=1}^{p}w_{l}^{i}
\varphi_{k}(s_{i})+v_{l}^{j_{p+1}}
\end{equation}
Hence, the behavior of all first order partial derivatives of $P_{k}f$ only
depends on the $k$th marginal distribution function $\varphi_{k}$. The latter
are again differentiable and it follows for $r\in D$, $r\neq k$, 
\begin{equation}\label{secondpartial}
\frac{\partial^{2} P_{k}f}{\partial\xi_{r}\partial\xi_{l}}(\xi^{k})
=\sum_{i=1}^{p}\frac{w_{l}^{i}w_{r}^{i}}{w_{k}^{i}}\rho_{k}(s_{i}).
\end{equation}
Hence, $P_{k}f$ is second order continuously differentiable on the neighborhood
$\B_{\epsilon_0}(\xi^{k}_0)$. More generally, if $\rho_{k}\in C^{s-2}(\R)$ for 
some $s\in\N$, $s\ge 2$, $P_{k}f$ is $s$-times continuously differentiable on 
the neighborhood $\B_{\epsilon_0}(\xi^{k}_0)$. 

For the case {\bf $P2$.)}, we consider $\xi^{k}_1\in\B_{\epsilon_0}(\xi^{k}_0)$ 
and the corresponding projection 
\[
(P_{k}f)(\xi^{k}_1)=\sum_{i=1}^{p(\xi^{k}_1)+1}\int_{a_{i-1}(\xi^{k}_1)}^{a_{i}(\xi^{k}_1)}
\langle v^{j_{i}},(\xi_{1, s}^{k},\bar{h})-Tx\rangle\rho_{k}(s)ds.
\]
Since, as mentioned above, there is a partition of the affine one-dimensional 
space $\{(\xi_{0,s}^{k},\bar{h})-Tx: s\in\R\}$ into intervals, each one contained 
in the interior of different cones, we can consider actually $\epsilon_0$ to be 
small enough such that the affine space $\{(\xi_{1,s}^{k},\bar{h})-Tx: s\in\R\}$ 
contains intervals each one contained in the interior set of the same mentioned 
cones. Moreover, due to (A5) and the finite cone decomposition we must have 
$\{(\xi_{1,s}^{k},\bar{h})-Tx: s\in (-\infty,a_{1}(\xi^{k}_1))\}\subset\cK^{\circ}_{j_1}$ 
and $\{(\xi_{1,s}^{k},\bar{h})-Tx: s \in (a_{p(\xi^{k}_1)}(\xi^{k}_1)),+\infty)\}\subset 
\cK^{\circ}_{j_{p(\xi^{k}_0)+1}}$.
Thus we can modify our notation to write 
\begin{eqnarray*}
(P_{k}f)(\xi^{k}_1)&=&
\int_{-\infty}^{a_{1,1}(\xi^{k}_1)}\langle v^{j_{1}},(\xi_{1, s}^{k},\bar{h})-Tx\rangle\rho_{k}(s)ds\\ 
&&+\sum_{r=1}^{m_1-1}\int_{a_{1,r}(\xi^{k}_1)}^{a_{1,r+1}(\xi^{k}_1)}\langle v^{j_{1,r}},(\xi_{1, s}^{k},\bar{h})-Tx\rangle\rho_{k}(s)ds \\ 
&&+ \int_{a_{1,m_1}(\xi^{k}_1)}^{a_{2,1}(\xi^{k}_1)}\langle v^{j_{2}},
(\xi_{1, s}^{k},\bar{h})-Tx\rangle\rho_{k}(s)ds\\ 
&&+\sum_{r=1}^{m_2-1}\int_{a_{2,r}(\xi^{k}_1)}^{a_{2,r+1}(\xi^{k}_1)}\langle v^{j_{2,r}},(\xi_{1, s}^{k},\bar{h})-Tx\rangle\rho_{k}(s)ds \\ 
&& \vdots \\
&&+\int_{a_{p(\xi^{k}_0)-1,m_{p(\xi^{k}_0)-1}} (\xi^{k}_1)}^{a_{p(\xi^{k}_0),1}  (\xi^{k}_1)} \langle v^{j_{p(\xi^{k}_0)}},(\xi_{1, s}^{k},\bar{h})-Tx\rangle\rho_{k}(s)ds\\ 
&&+\sum_{r=1}^{m_{p(\xi^{k}_0)}-1}\int_{a_{p(\xi^{k}_0),r}(\xi^{k}_1)}^{a_{p(\xi^{k}_0),r+1}(\xi^{k}_1)} \langle v^{j_{p(\xi^{k}_0),r}},(\xi_{1, s}^{k},\bar{h})-Tx\rangle\rho_{k}(s)ds \\ 
&& +\int_{a_{p(\xi^{k}_0),1}(\xi^{k}_1)}^{+\infty}\langle 
v^{j_{p(\xi^{k}_0)+1}},(\xi_{1, s}^{k},\bar{h})-Tx\rangle\rho_{k}(s)ds, 
\end{eqnarray*}
where for each vertex $v^{j_{i,t}}$, $1\le t \le m_i$, we have a corresponding cone $\cK_{i,t} \in \varLambda_{s_i}$, and for simplicity we omitted in the notation the dependence of $m_i$ on $\xi^{k}_1$. The functions $a_{i,t}$, $1\le i\le p(\xi^{k}_0)$, $1\le t \le m_i$, are affine since they can be obtained through equation \eqref{si} by considering the corresponding neighboring cones belonging to 
$\cK_{s_i}$,  $1\le i \le p(\xi^{k}_0)$. Moreover, we have $a_{i,1}(\xi^{k}_0) =a_{i,m_i}(\xi^{k}_0)=s_{i}(\xi^{k}_0)$, $1\le i \le p(\xi^{k}_0)$. 
Note also that in this representation we can have $m_i=1$ for some 
$1\le i \le p(\xi^{k}_0)$, meaning  in this case that the corresponding sum of integral terms vanishes, and we only have to consider the corresponding 
limiting function $a_{i,1}$. \\
We show now the existence of partial derivatives $\frac{\partial (P_{k}f)(.)}{\partial\xi_{l}}$, $l\neq k$, at $\xi^{k}_0 $. Because no facet of the cones
is parallel to the canonical basis element $e_{l}$, we have that there exists 
$\epsilon_{l^+}$ such that $\cK(\xi^{k})=\cK(\xi^{k}_0 + e_l \epsilon_{l^+})$,
for all elements $\xi^{k}$ in the line segment $(\xi^{k}_0,\xi^{k}_0 + e_l \epsilon_{l^+}] $.  Thus, there exist corresponding continuous limiting functions, denoted by $b_{i,j}$,  $\le i \le p(\xi^{k}_0), 1\le j \le m_i$, 
that are affine, defined on $(\xi^{k}_0,\xi^{k}_0 + e_l \epsilon_{l^+}]$, and 
for which the derivative exist on the open segment $(\xi^{k}_0,\xi^{k}_0 + e_l \epsilon_{l^+})$. By using the univariate mean-value theorem we have 
for $0< \epsilon \le \epsilon_{l^+}$
\begin{eqnarray*}
&&(P_{k}f)(\xi^{k}_0 + e_l\epsilon)- (P_{k}f)(\xi^{k}_0)=\\
&&\Big( \int_{-\infty}^{b_{1,1}(\xi^{k}_0 +   e_l \mu)} v^{j_{1}}_l\rho_{k}(s)ds +
\sum_{r=1}^{m_1-1}\int_{b_{1,r}(\xi^{k}_0 +   e_l \mu)}^{b_{1,r+1}(\xi^{k}_0 +   e_l \mu)} v^{j_{1,r}}_l \rho_{k}(s)ds+
\end{eqnarray*}
\begin{eqnarray*} 
&& 
+\int_{b_{1,m_1}(\xi^{k}_0 + e_l \mu)}^{b_{2,1}(\xi^{k}_0 +   e_l \mu)} v^{j_{2}}_l\rho_{k}(s)ds +\sum_{r=1}^{m_2-1}\int_{b_{2,r}(\xi^{k}_0 +   e_l \mu)}^{b_{2,r+1}(\xi^{k}_0 +   e_l \mu)} v^{j_{2,r}}_l\rho_{k}(s)ds \\ 
& & \vdots \\
& &
+\int_{b_{p(\xi^{k}_0)-1, m_{p(\xi^{k}_0)-1}}(\xi^{k}_0 + e_l \mu)}^{b_{p(\xi^{k}_0),1}(\xi^{k}_0 + e_l \mu)}\!\!\! v^{j_{p(\xi^{k}_0)}}_l\rho_{k}(s)ds +
 \!\!\!\sum_{r=1}^{m_{p(\xi^{k}_0)}-1}\!\!\!\int_{b_{p(\xi^{k}_0),r}
 (\xi^{k}_0 + e_l \mu)}^{b_{p(\xi^{k}_0),r+1}
 (\xi^{k}_0 + e_l \mu)}\!\!\! v^{j_{p(\xi^{k}_0),r}}_l\rho_{k}(s)ds \\ 
&& +
\int_{b_{p(\xi^{k}_0), m_{p(\xi^{k}_0)}}(\xi^{k}_0 + e_l \mu)}^{+\infty}v^{j_{p(\xi^{k}_0)+1}}_l\rho_{k}(s)ds\Big)\;\epsilon,
\end{eqnarray*}
for some $0< \mu \le \epsilon$. We also have that 
\begin{eqnarray*}
&&\left|\sum_{r=1}^{m_i-1}\int_{b_{i,r}(\xi^{k}_0 + e_l \mu)}^{b_{i,r+1}(\xi^{k}_0 + e_l \mu)} v^{j_{i,r}}_l\rho_{k}(s)ds \right|\le
\int_{b_{i,1}(\xi^{k}_0 + e_l \mu)}^{b_{i,m_i}(\xi^{k}_0 + e_l \mu)} 
\max_{j: \; \cK_j \in \cK(\xi_{0}^{k})}|v^{j}_l|\rho_{k}(s)ds\\
&&\le \max_{j: \; \cK_j \in \cK(\xi_{0}^{k})}|v^{j}_l| \max_{s \in [s_a, s_b ]}
\rho_{k}(s)\left( b_{i,m_i}(\xi^{k}_0 + e_l \mu)-b_{i,1}(\xi^{k}_0 + e_l \mu) \right),
\end{eqnarray*}
where $\displaystyle s_a:= \min_{\xi \in [\xi^{k}_0,\xi^{k}_0 +   e_l \epsilon_{l^+}] }b_{i,1}(\xi)$ and 
$\displaystyle s_b:= \max_{\xi \in [\xi^{k}_0,\xi^{k}_0 +   e_l \epsilon_{l^+}] }b_{i,m_i}(\xi)$.\\
Because $b_{i,1}(\xi^{k}_0)=b_{i,m_i}(\xi^{k}_0)=s_{i}(\xi^{k}_0)$, $1\le i \le p(\xi^{k}_0)$, we can divide both sides of the above equality by $\epsilon$, 
and then take $\epsilon \downarrow 0$ to obtain
\begin{eqnarray*}
\frac{\partial (P_{k}f)(\xi^{k}_0)^+}{\partial\xi_{l}}&=&
 \int_{-\infty}^{s_{1}(\xi^{k}_0)} v^{j_{1}}_k\rho_{k}(s)ds +
\int_{s_{1}(\xi^{k}_0)}^{s_{2}(\xi^{k}_0)} v^{j_{2}}_k\rho_{k}(s)ds+ \\
& &\! \vdots \\
& &\!
 +\int_{s_{p(\xi^{k}_0)-1} (\xi^{k}_0)}^{s_{p(\xi^{k}_0)}(\xi^{k}_0)} v^{j_{p(\xi^{k}_0)}}_k\rho_{k}(s)ds +
\int_{s_{p(\xi^{k}_0)}(\xi^{k}_0)}^{+\infty}v^{j_{p(\xi^{k}_0)+1}}_k\rho_{k}(s)ds.
\end{eqnarray*}
A similar argument in the opposite direction, that is on the segment 
$(\xi^{k}_0,\xi^{k}_0 - e_l \epsilon_{l^-})$, shows that in fact we obtain
\[
\frac{\partial (P_{k}f)(\xi^{k}_0)}{\partial\xi_{l}}=\sum_{i=1}^{p(\xi^{k}_0)+1}\int_{s_{i-1}(\xi^{k}_0)}^{s_{i}(\xi^{k}_0)} v^{j_{i}}_k\rho_{k}(s)ds
=-\sum_{i=1}^{p(\xi_{0}^{k})}(v_{l}^{j_{i+1}}-v_{l}^{j_{i}})\varphi_{k}
(s_{i}(\xi_{0}^{k}))+v_{l}^{j_{p+1}} 
\]
and, hence, the same representation as in (\ref{firstpartial}).
Because this argument is valid for each point $\xi^{k}_0\in\R^{d-1}$, we have that all first order partial derivatives of $P_{k}f$ exist at each point of 
$\R^{d-1}$. Note that partial differentiability with respect to $\xi_{k}$
holds by definition.\\
To prove that a partial derivative $\frac{\partial (P_{k}f)(.)}{\partial\xi_{l}} , \: l\neq k, $ is continuous  at $\xi^{k}_0$ for the case {\bf $P2$.)}, we consider 
a sequence of points $(\xi^{k}_n)_{n\in\N} \in \B_{\epsilon_0}(\xi^{k}_0)$ converging to $\xi^{k}_0$. Then for each point 
$\xi^{k}_n \in \B_{\epsilon_0}(\xi^{k}_0)$ we have
\begin{eqnarray*}
\frac{\partial (P_{k}f)(\xi^{k}_n)}{\partial\xi_{l}}&=&
\int_{-\infty}^{a_{1,1}(\xi^{k}_n)} v^{j_{1}}_l\rho_{k}(s)ds +
\sum_{r=1}^{m_1(\xi^{k}_n)-1}\int_{a_{1,r}(\xi^{k}_n)}^{a_{1,r+1}(\xi^{k}_n)}  
v^{j_{1,r}}_l \rho_{k}(s)ds \\ 
&& +\int_{a_{1,m_1(\xi^{k}_n)}(\xi^{k}_n)}^{a_{2,1}(\xi^{k}_n)} v^{j_{2}}_l
\rho_{k}(s)ds +\sum_{r=1}^{m_2(\xi^{k}_n)-1}\int_{a_{2,r}(\xi^{k}_n)}^{a_{2,r+1}(\xi^{k}_n)} v^{j_{2,r}}_l\rho_{k}(s)ds +\\ 
&& \vdots \\
&&+\int_{a_{p_0-1, m_{p_0-1}} (\xi^{k}_n)}^{a_{p_0,1}(\xi^{k}_n)} v^{j_{p_0}}_l\rho_{k}(s)ds +\!\!\sum_{r=1}^{m_{p_0}(\xi^{k}_n) -1}\int_{a_{p_0,r}
(\xi^{k}_n)}^{a_{p_0,r+1}(\xi^{k}_n)}\! v^{j_{p_0,r}}_l\rho_{k}(s)ds \\ 
&& +\int_{a_{p_0, m_{p_0}(\xi^{k}_n)}
(\xi^{k}_n)}^{+\infty}v^{j_{p_0+1}}_l\rho_{k}(s)ds ,
\end{eqnarray*}
where we introduced the short notation $p_0:=p(\xi^{k}_0)$.  Note that if we 
have a limiting function $a(.)$ obtained trough equation \eqref{si} by two 
adjacent cones $K_{j_r},K_{j_t} \in \varLambda_{s_i} $, then we have that 
$a(.)$ is affine, and $a(\xi^{k}_0)=s_i(\xi^{k}_0)$. Moreover, by considering 
that are finite many different cones contained in $\varLambda_{s_i}$, it is 
clear that we can have at most finite many possible different limiting functions $a(.)$ that can be obtained from two adjacent cones $K_{j_r},K_{j_t}\in \varLambda_{s_i} $ by  \eqref{si}. Let us denote by $\overline{a}_{s_i}(.)$ 
the maximum, and by $\underline{a}_{s_i}(.)$ the minimum, of all such limiting functions $a(.)$ over $\B_{\epsilon_0}(\xi^{k}_0)$. Then we have that 
$\overline{a}_{s_i}(.)$ and $\underline{a}_{s_i}(.)$ are continuous on 
$\B_{\epsilon_0}(\xi^{k}_0)$. We also have that $\overline{a}_{s_i}(\xi^{k}_0)=\underline{a}_{s_i}(\xi^{k}_0)=s_i(\xi^{k}_0)$. 
Thus we have
\begin{eqnarray*}
&&\left|\sum_{r=1}^{m_i(\xi^{k}_n)-1}\int_{a_{i,r}(\xi^{k}_n)}^{a_{i,r+1}(\xi^{k}_n)} v^{j_{1,r}}_l \rho_{k}(s)ds \right|\le
\int_{\underline{a}_{s_i}(\xi^{k}_n)}^{\overline{a}_{s_i}(\xi^{k}_n)} 
\max_{j: \; \cK_j \in \cK(\xi_{0}^{k})}|v^{j}_l|\rho_{k}(s)ds\\
&&\le \max_{j: \cK_j \in \cK(\xi_{0}^{k})}|v^{j}_l| \max_{s \in [s_{i,a},s_{i,b}]}
\rho_{k}(s) \left( \overline{a}_{s_i}(\xi^{k}_n) - \underline{a}_{s_i}(\xi^{k}_n) \right),
\end{eqnarray*}
where $\displaystyle s_{i,a}:= \min_{\xi \in \overline{B}_{\epsilon_0}(\xi^{k}_0) } \underline{a}_{s_i}(\xi)$,  and 
$\displaystyle s_{i,b}:= \max_{\xi \in \overline{B}_{\epsilon_0}(\xi^{k}_0) } \overline{a}_{s_i}(\xi)$.\\
By letting $\xi^{k}_n \rightarrow \xi^{k}_0$ the the right-hand side of the 
latter inequality tends to zero. This holds for $1\le i \le p(\xi^{k}_0)$. Therefore, we obtain
\[
\lim_{n\rightarrow \infty} \frac{\partial (P_{k}f)(\xi^{k}_n)}{\partial\xi_{l}}=
\sum_{i=1}^{p(\xi^{k}_0) +1}\int_{s_{i-1}(\xi^{k}_0)}^{s_{i}(\xi^{k}_0)} v^{j_{i}}_k\rho_{k}(s)ds=
 \frac{\partial (P_{k}f)(\xi^{k}_0)}{\partial\xi_{l}},
\]
which proves continuity of $\frac{\partial (P_{k}f)()}{\partial\xi_{l}}$ under the case {\bf $P2$.)} . \\
By combining both {\bf $P1$.)} and {\bf $P2$.)} $P_{k}f$ is continuously differentiable on $\R^{d}$.
\hfill$\Box$
\end{proof}

\begin{corollary}\label{c1}
Let $\emptyset\neq u\subseteq D$ and $x\in X$. Assume (A1)--(A5).
Then the projection $P_{u}f$ is continuously differentiable on $\R^{d}$ and
second order continuously differentiable almost everywhere in $\R^{d}$.
\end{corollary}

\begin{proof}
If $|u|=1$ the result follows from Theorem \ref{t2}. For
$u=\{k,j\}$ with $k,j\in D$, $k\neq j$, we obtain from the Leibniz
theorem \cite[Theorem 1]{GrKS13} for $l\not\in u$ and $r\not\in u$
\begin{eqnarray*}
D_{l}P_{u}f(\xi^{u})&:=&\frac{\partial}{\partial\xi_{l}}P_{u}f(\xi^{u})=
P_{j}\frac{\partial}{\partial\xi_{l}}P_{k}f(\xi^{u})\\
D_{r}D_{l}P_{u}f(\xi^{u})&:=&\frac{\partial^{2}}{\partial\xi_{l}
\partial\xi_{r}}P_{u}f(\xi^{u})=P_{j}\frac{\partial^{2}}{\partial\xi_{l}
\partial\xi_{r}}P_{k}f(\xi^{u})
\end{eqnarray*}
and from the proof of Theorem \ref{t2}
\begin{eqnarray}\label{firstpartPu}
D_{l}P_{u}f(\xi^{u})&=&-\sum_{i=1}^{p}w_{l}^{i}\int_{\R}\varphi_{k}
(s_{i}(\xi^{k}))\rho_{j}(\xi_{j})d\xi_{j}+v_{l}^{j_{p+1}}\\
D_{r}D_{l}P_{u}f(\xi^{u})&=&\sum_{i=1}^{p}\frac{w_{l}^{i}w_{r}^{i}}{w_{k}^{i}}
\int_{\R}\rho_{k}(s_{i}(\xi^{k}))\rho_{j}(\xi_{j})d\xi_{j}.
\label{secpartPu}
\end{eqnarray}
If $u$ contains more than two elements, the integrals on the right-hand 
side become multiple integrals. In all cases, however, such an integral 
is a continuous function of the remaining variables $\xi_{i}$, 
$i\in D\setminus u$. This can be shown using Lebesgue's theorem as 
$\varphi_{k}$ and $\rho_{k}$ are continuous and bounded on $\R$. 
\hfill$\Box$
\end{proof}

The following is the main result of this section.
\begin{theorem}\label{t3}
Assume (A1)--(A5). Then all ANOVA terms of $f$ except the one of highest 
order are first order continuously differentiable on $\R^{d}$ and all second 
order partial derivatives exist are continuous except in a set of Lebesgue measure 
zero and quadratically integrable with respect to the density $\rho$. In particular, the first and second order ANOVA terms of $f$ belong to the 
tensor product Sobolev space $\cW_{2,{\rm mix}}^{(1,\ldots,1)}(\R^{d})$.
\end{theorem}

\begin{proof}
According to (\ref{anova}) the ANOVA terms of $f$ are defined by
\[
f_{u}=P_{-u}(f)+\sum_{v\subsetneq u}(-1)^{|u|-|v|}P_{-v}(f)
\]
for all nonempty subsets $u$ of $D$. Hence, all ANOVA terms of $f$ for $u\neq D$
are continuously differentiable on $\R^{d}$. Second order partial derivatives
of those ANOVA terms exist and are continuous at least at those $\xi$ such that 
$(\xi,\bar{h})-Tx$ belongs to the interior of some cone $\cK_{j}$, i.e., almost
everywhere in $\R^{d}$. The non-vanishing first order partial derivatives of the second order ANOVA terms are of the form
\begin{eqnarray*}
D_{l}f_{\{l,r\}}(\xi_{l},\xi_{r})&=&D_{l}P_{D\setminus\{l,r\}}f(\xi_{l},\xi_{r})-D_{l}P_{D\setminus\{l\}}f(\xi_{l})\\
&=&-\sum_{i=1}^{p}w_{l}^{i}\int_{\R}\varphi_{k}(s_{i}(\xi^{k}))\!\prod_{i\in D
\setminus\{l,r\}\atop{i\neq k}}\!\rho_{i}(\xi_{i})d\xi^{-\{l,r\}}
-D_{l}P_{D\setminus\{l\}}f(\xi_{l})
\end{eqnarray*}
for all $l,r\in D$ and some $k\in D$. Since $\varphi_{k}$ is Lipschitz continuous,
the function $D_{l}f_{\{l,r\}}(\xi_{l},\cdot):\R\to\R$ is Lipschitz continuous,
too, and, hence, partially differentiable with respect to $\xi_{r}$ in the sense
of Sobolev (see, for example, \cite[Section 4.2.3]{EvGa92}). Furthermore, the
second order partial derivative is a bounded function (see also (\ref{secondpartial})) and due to (A3) quadratically integrable with respect to 
$\rho$. \hfill$\Box$
\end{proof}

\begin{remark}\label{r1}
The second order ANOVA approximation of $f$, i.e.,
\begin{equation}\label{secordanova}
f^{(2)}:=\sum_{|u|=1\atop{u\subseteq D}}^{2}f_{u}
\end{equation}
belongs to the tensor product Sobolev space $\cW_{2,{\rm mix}}^{(1,\ldots,1)}(\R^{d})$. 
Hence, if the effective superposition dimension is at most 2, $f^{(2)}$ is a good 
approximation of $f$ due to (\ref{errdim}) and favorable behavior of randomly shifted 
lattice rules may be expected.
\end{remark}

The following two examples show that conditions (A1)--(A5) are necessary for 
the first order continuous differentiability of projections, but, in general, 
do not imply continuity of second order partial derivatives of the projections.

\begin{example}\label{ex3}
Let $\bar{m}=3$, $d=2$, $\Xi=\R^{2}$, $P$ denote a probability distribution with independent marginal densities $\rho_i$, $i=1,2$, whose means are w.l.o.g. equal 
to $0$. We assume that (A3) is satisfied for $P$. Let the vector $q$ and matrix $W$
\[
W=\left(\begin{array}{cccc} -1 & 1 & 0 \\ 1 & 1 & -1
\end{array}\right)\qquad  q=\left(\begin{array}{c} 1 \\ 1 \\ 0
\end{array}\right)
\]
be given. Then (A1) and (A2) are satisfied and the dual feasible set
$\cD$ is
\[
\cD=\{z\in\R^{2}:W^{\top}z\le q\}=\{z\in\R^{2}:-z_{1}+z_{2}\le 1,
z_{1}+z_{2}\le 1, -z_{2}\le 0\},
\]
i.e., $\cD$ is a triangle and has the three vertices
\[
v^{1}=\left(\begin{array}{c} 1 \\ 0 \end{array}\right)\quad
v^{2}=\left(\begin{array}{c} -1 \\ 0 \end{array}\right)\quad
v^{3}=\left(\begin{array}{c} 0 \\ 1 \end{array}\right).
\]
\begin{figure}[ht]
\begin{center}
\begin{picture}(200,150)
\put(0,50){\line(1,0){200}} \put(100,50){\line(0,1){100}}
\textcolor{blue}{\put(100,0){\line(0,1){50}}
\put(100,50){\line(-1,1){90}} \put(100,50){\line(1,1){90}}}
\put(100,50){\circle*{2}} \put(60,50){\circle*{2}}
\put(100,90){\circle*{2}} \put(140,50){\circle*{2}} \thicklines
\put(60,50){\line(1,1){40}} \put(100,90){\line(1,-1){40}}
\put(60,50){\line(1,0){80}} \put(40,30){\makebox(0,0)[t]{$\cK_{2}$}}
\put(160,30){\makebox(0,0)[t]{$\cK_{1}$}}
\put(120,110){\makebox(0,0)[t]{$\cK_{3}$}}
\put(95,45){\makebox(0,0)[t]{$0$}}
\put(108,100){\makebox(0,0)[t]{$v^{3}$}}
\put(60,45){\makebox(0,0)[t]{$v^{2}$}}
\put(140,45){\makebox(0,0)[t]{$v^{1}$}}
\put(120,60){\makebox(0,0)[t]{$\cD$}}
\end{picture}
\caption{Illustration of $\cD$, its vertices $v^{j}$ and the normal
cones $\cK_{j}$ to its vertices}
\end{center}
\label{fig1}
\end{figure}
Hence, the second component of the two adjacent vertices  $v^{1}$
and $v^{2}$ coincides. According to (\ref{nc}) the normal cones
$\cK_{j}$ to $\cD$ at $v^{j}$, $j=1,2,3$, are
\begin{eqnarray*}
\cK_{1}&=&\{z\in\R^{2}:z_{1}\ge 0,z_{2}\le z_{1}\},\quad
\cK_{2}=\{z\in\R^{2}:z_{1}\le 0,z_{2}\le -z_{1}\},\\
\cK_{3}&=&\{z\in\R^{2}:z_{2}\ge z_{1},z_{2}\ge -z_{1}\}.
\end{eqnarray*}
The function $\Phi$ (see (\ref{int2st})) is of the form
\[
\phi(t)=\max_{i=1,2,3}\langle v^{i},t
\rangle=\max\{t_{1},-t_{1},t_{2}\} =\max\{|t_{1}|,t_{2}\}
\]
and the two-stage stochastic program is
\begin{equation}\label{sp2}
\min\Big\{\langle c, x\rangle+\int_{\R^{2}}\max\{|\xi_{1}-[Tx]_{1}|,
\xi_{2}-[Tx]_{2}\}\rho(\xi)d\xi:x\in X\Big\}.
\end{equation}
The ANOVA projection $P_{1}f$ is defined by
\[
(P_{1}f)(\xi_{2})=\int_{-\infty}^{+\infty}\!\!\max\{|\xi_{1}-[Tx]_{1}|,
\xi_{2}-[Tx]_{2}\}\rho_{1}(\xi_{1})d\xi_{1}\quad(\xi_{2}\in\R).
\]
For $\xi_{2}-[Tx]_{2}\le 0$ one obtains
\begin{eqnarray*}
(P_{1}f)(\xi_{2})&=&\int_{-\infty}^{+\infty}\!\!|\xi_{1}-[Tx]_{1}|\rho_{1}(\xi_{1})d\xi_{1}\\
&=&\int_{-\infty}^{+\infty}(\xi_{1}-[Tx]_{1})\rho_{1}(\xi_{1})d\xi_{1}-2\int_{-\infty}^{[Tx]_{1}}
(\xi_{1}-[Tx]_{1})\rho_{1}(\xi_{1})d\xi_{1}
\end{eqnarray*}
and in case $\xi_{2}-[Tx]_{2}\ge 0$ 
\begin{eqnarray*}
(P_{1}f)(\xi_{2})&=&\!\int_{-\infty}^{+\infty}\!\!\!|\xi_{1}-[Tx]_{1}|\rho_{1}(\xi_{1})d\xi_{1}-\!
\int_{0}^{\xi_{2}-[Tx]_{2}}\!\!\!\!\!(\xi_{1}+\xi_{2}-[Tx]_{1}-[Tx]_{2})\rho_{1}(\xi_{1})d\xi_{1}.
\end{eqnarray*}
Hence, $P_{1}f$ belongs to $C^{1}(\R)$ for all $x\in X$ if $\rho$ is continuous.\\
When calculating the ANOVA projection $P_{2}f$, notice that assumption (A5) 
is violated. We obtain
\[
(P_{2}f)(\xi_{1})=|\xi_{1}-[Tx]_{1}|\int_{-\infty}^{|\xi_{1}-[Tx]_{1}|}\rho_{2}(\xi_{2})d\xi_{2}+
\int_{|\xi_{1}-[Tx]_{1}|}^{+\infty}(\xi_{2}-[Tx]_{2})\rho_{2}(\xi_{2})d\xi_{2}
\]
and {\em $P_{2}f$ does not belong to $C^{1}(\R)$} for all $x\in X$.
\end{example}

\begin{example}\label{ex4}
Let $\bar{m}=3$, $d=2$, $P$ denote a two-dimensional probability distribution 
with independent continuous marginal densities $\rho_i$, $i=1,2$, whose means 
are w.l.o.g. equal to $0$. Again we assume that (A3) is satisfied for $P$. Let 
the vector $q$ and matrix $W$
\[
W=\left(\begin{array}{cccc} -1 & 1 & 1 \\ 1 & 1 & 3
\end{array}\right)\qquad  q=\left(\begin{array}{c} 1 \\ 1 \\ -1
\end{array}\right)
\]
be given. Then (A1) and (A2) are satisfied and the dual feasible set
$\cD$ is
\[
\cD=\{z\in\R^{2}:W^{\top}z\le q\}=\{z\in\R^{2}:-z_{1}+z_{2}\le 1,
z_{1}+z_{2}\le 1, z_{1}+3z_{2}\le -1\},
\]
i.e., $\cD$ is also a triangle and has the three vertices
\[
v^{1}=\left(\begin{array}{c} 2 \\ -1 \end{array}\right)\quad
v^{2}=\left(\begin{array}{c} -1 \\ 0 \end{array}\right)\quad
v^{3}=\left(\begin{array}{c} 0 \\ 1 \end{array}\right).
\]
Hence, both components of the vertices  $v^{j}$, $j=1,2,3$, are distinct.
This means that (A4) and (A5) are satisfied. The normal cones $\cK_{j}$ 
to $\cD$ at $v^{j}$, $j=1,2,3$, are
\begin{eqnarray*}
\cK_{1}&=&\{z\in\R^{2}:z_{1}\ge z_{2},z_{1}\ge 3z_{2}\},\quad
\cK_{2}=\{z\in\R^{2}:z_{1}\le 3z_{2},z_{2}\le -z_{1}\},\\
\cK_{3}&=&\{z\in\R^{2}:z_{2}\ge z_{1},z_{2}\ge -z_{1}\}.
\end{eqnarray*}

\begin{figure}[ht]
\begin{center}
\begin{picture}(200,150)
\put(0,50){\line(1,0){200}} \put(100,0){\line(0,1){150}}
\textcolor{blue}{\put(83,0){\line(1,3){17}}
\put(100,50){\line(-1,1){90}} \put(100,50){\line(1,1){90}}}
\put(100,50){\circle*{2}} \put(60,50){\circle*{2}}
\put(100,90){\circle*{2}} \put(180,10){\circle*{2}} 
\thicklines
\put(60,50){\line(1,1){40}} \put(100,90){\line(1,-1){80}}
\put(60,50){\line(3,-1){120}} 
\textcolor{blue}{
\put(40,35){\makebox(0,0)[t]{$\cK_{2}$}}
\put(185,35){\makebox(0,0)[t]{$\cK_{1}$}}
\put(120,110){\makebox(0,0)[t]{$\cK_{3}$}}}
\put(93,46){\makebox(0,0)[t]{$0$}}
\put(108,100){\makebox(0,0)[t]{$v^{3}$}}
\put(60,45){\makebox(0,0)[t]{$v^{2}$}}
\put(190,15){\makebox(0,0)[t]{$v^{1}$}}
\put(120,43){\makebox(0,0)[t]{$\cD$}}
\put(198,45){\makebox(0,0)[t]{$z_{1}$}}
\put(95,148){\makebox(0,0)[t]{$z_{2}$}}
\end{picture}
\vspace{0,8cm}
\caption{Illustration of $\cD$, its vertices $v^{j}$ and the normal
cones $\cK_{j}$ to its vertices}
\end{center}
\label{fig2}
\end{figure}
The function $\phi$ is of the form
\[
\phi(t)=\max_{i=1,2,3}\langle v^{i},t
\rangle=\max\{2t_{1}-t_{2},-t_{1},t_{2}\} 
\]
and the two-stage stochastic program is
\begin{equation}\label{sp3}
\min\Big\{\langle c, x\rangle+\int_{\R^{2}}\phi(\xi_{1}-[Tx]_{1},\xi_{2}-[Tx]_{2})
\rho_{1}(\xi_{1})\rho_{2}(\xi_{2})d\xi_{1}d\xi_{2}:x\in X\Big\}.
\end{equation}
Then its ANOVA projection $P_{1}f$ is given by
\[
(P_{1}f)(\xi_{2})=\int_{-\infty}^{+\infty}\max\{2(s-[Tx]_{1})-\xi_{2}+[Tx]_2,
-s+[Tx]_{1},\xi_{2}-[Tx]_{2}\}\rho_{1}(s)ds
\]
for every $\xi_{2}\in\R$. For simplicity let $x=0$. First let $\xi_{2}>0$.
\begin{eqnarray*}
(P_{1}f)(\xi_{2})&=&\int_{-\infty}^{+\infty}\max\{2s-\xi_{2},-s,\xi_{2}\}
\rho_{1}(s)ds\\
&=&\int_{-\infty}^{s_{1}}-s\rho_{1}(s)ds+\int_{s_{1}}^{s_{2}}\xi_{2}\rho_{1}(s)ds
+\int_{s_{2}}^{+\infty}(2s-\xi_{2})\rho_{1}(s)ds,
\end{eqnarray*}
where $s_{1}=s_{1}(\xi_{2})=-\xi_{2}$ and $s_{2}=s_{2}(\xi_{2})=\xi_{2}$. Hence,
\begin{eqnarray*}
(P_{1}f)(\xi_{2})&=&3\int_{\xi_{2}}^{+\infty}s\rho_{1}(s)ds+\xi_{2}\Big(
\int_{-\xi_{2}}^{\xi_{2}}\rho_{1}(s)ds-\int_{\xi_{2}}^{+\infty}\rho_{1}(s)ds\Big).
\end{eqnarray*}
Now, we compute the partial derivatives for $\xi_{2}>0$ and obtain
\begin{eqnarray*}
\frac{\partial P_{1}f}{\partial\xi_{2}}(\xi_{2})&=&-\xi_{2}(\rho_{1}(\xi_{2})
-\rho_{1}(-\xi_{2}))+(2\varphi_{1}(\xi_{2})-\varphi_{1}(-\xi_{2})-1)\\
\frac{\partial P_{1}f}{\partial\xi_{2}}(0+)&=&\varphi_{1}(0)-1\\
\frac{\partial^{2} P_{1}f}{\partial\xi_{2}^{2}}(\xi_{2})&=&
2\rho_{1}(\xi_{2})+\rho_{1}(-\xi_{2}) \,=\, 3\rho_{1}(\xi_{2})
\end{eqnarray*}
$P_{1}f$ is for $\xi_{2}>0$ $s$-times continuously differentiable if 
$\rho_{1}\in C^{s-2}(\R)$ for any $s\in\N$.\\
Now, let $\xi_{2}< 0$. Then we obtain with $s_{1}(\xi_{2})=\frac{\xi_{2}}{3}$
\begin{eqnarray*}
(P_{1}f)(\xi_{2})&=&\int_{-\infty}^{+\infty}\max\{2s-\xi_{2},-s,\xi_{2}\}
\rho_{1}(s)ds\\
&=&\int_{-\infty}^{s_{1}}-s\rho_{1}(s)ds+\int_{s_{1}}^{+\infty}(2s-\xi_{2})
\rho_{1}(s)ds\\
\frac{\partial P_{1}f}{\partial\xi_{2}}(\xi_{2})&=&-\textstyle\frac{\xi_{2}}{3}\rho_{1}(\textstyle\frac{\xi_{2}}{3})
+(\varphi_{1}(\textstyle\frac{\xi_{2}}{3})-1)+\textstyle\frac{\xi_{2}}{3}\rho_{1}(\textstyle\frac{\xi_{2}}{3})=\varphi_{1}(\textstyle\frac{\xi_{2}}{3})-1\\
\frac{\partial P_{1}f}{\partial\xi_{2}}(0-)&=&\varphi_{1}(0)-1\\
\frac{\partial^{2} P_{1}f}{\partial\xi_{2}^{2}}(\xi_{2})&=&\textstyle\frac13\rho_{1}(\textstyle\frac{\xi_{2}}{3})
\end{eqnarray*}
$P_{1}f$ is for $\xi_{2}<0$ $s$-times continuously differentiable if 
$\rho_{1}\in C^{s-2}(\R)$ for any $s\in\N$.\\
Hence, {\em $P_{1}f$ belongs to $C^{1}(\R)$}, but its second derivative is discontinuous at $\xi_{2}=0$. The same holds for $P_{2}f$.
\end{example}

\begin{remark}\label{r2} (error estimate)\\
If the assumptions of Theorem \ref{t3} are satisfied and all marginal 
densities $\rho_{j}$, $j\in D$, are positive, all ANOVA terms $g_{u}$, 
$|u|=1,\,2$, of $g$ given by (\ref{anova01}) belong to the tensor product
Sobolev space (\ref{wsob}). Then the QMC quadrature error may be estimated
as follows:
\begin{eqnarray}\nonumber
\Big|\int_{\R^{d}}f(\xi)\rho(\xi)d\xi - n^{-1}\sum_{j=1}^{n}f(\xi^{j})\Big|
&=&\Big|\int_{[0,1]^{d}}g(x)dx - n^{-1}\sum_{j=1}^{n}g(x^{j})\Big|\\
\nonumber
&\leq&\sum_{0<|u|\le d}\Big|\int_{[0,1]^{d}}g_{u}(x^{u})dx^{u}-
n^{-1}\sum_{j=1}^{n}g_{u}(x^{j})\Big|\\ \label{est3}
&\leq&\sum_{|u|=1}^{2}{\rm Disc}_{n,u}(x^{1},\ldots,x^{n})\|g_{u}\|_{\gamma}+\\ \label{est4}
&&\sum_{|u|=3}^{d}\Big|\int_{[0,1]^{d}}g_{u}(x)dx - n^{-1}\sum_{j=1}^{n}g_{u}(x^{j})\Big|,
\end{eqnarray}
where $x_{i}^{j}=\varphi_{i}(\xi_{i}^{j})\in(0,1)^{d}$, $j=1,\ldots,n$, 
$i=1,\ldots,d$, are the QMC points and ${\rm Disc}_{n,u}$ is the weighted 
$L_{2}$- discrepancy
\[
{\rm Disc}_{n,u}^{2}(x^{1},\ldots,x^{n})
=\gamma_{u}\int_{[0,1]^{|u|}}{\rm disc}_{u}^{2}(x^{u})dx^{u},
\]
where the discrepancy ${\rm disc}$ is given by
\[
{\rm disc}_{u}(x^{u})=\prod_{i\in u}x_{i}-n^{-1}
\big|\{j\in\{1,\ldots,n\}:x^{j}\in[0,x^{u})\}\big|,
\]
and $\|g_{u}\|_{\gamma}$ the weighted norm of $g_{u}$ given by
(\ref{wnorm}) in the weighted tensor product Sobolev space (\ref{wsob}).
Recalling the arguments in the introduction one may conclude that
all terms in (\ref{est3}) converge with the optimal rate (\ref{rate})
while all terms in (\ref{est4}) also converge to $0$ due to Proinov's
convergence result \cite{Proi88} (as the $g_{u}$ are continuous). In
addition, the sum (\ref{est4}) can be further estimated by
\begin{equation}\label{errterm}
\sum_{|u|=3}^{d}\Big(\int_{[0,1]^{d}}g_{u}^{2}(x)dx+n^{-1}
\sum_{j=1}^{n}g_{u}^{2}(x^{j})\Big)=\sum_{|u|=3}^{d}\Big(
\|f_{u}\|_{L_{2}}^{2}+n^{-1}\sum_{j=1}^{n}f_{u}^{2}(\xi^{j})\Big).
\end{equation}
Since (\ref{errdim}) implies $\sum_{|u|=3}^{d}\|f_{u}\|_{L_{2}}^{2}\leq
\varepsilon\sigma^{2}(f)$ if $d_{S}(\varepsilon)\leq 2$ and the second term 
on the right-hand side of (\ref{errterm}) represents a QMC approximation 
of the first term, we may conclude that the term in (\ref{est4}) is of the 
form $O(\varepsilon)$. Hence, we obtain the estimate
\begin{equation}\label{errorrate}
\Big|\int_{\R^{d}}f(\xi)\rho(\xi)d\xi - n^{-1}\sum_{j=1}^{n}f(\xi^{j})\Big|\leq
C(\delta)n^{-1+\delta}+O(\varepsilon)
\end{equation}
if the condition $d_{S}(\varepsilon)\leq 2$ is satisfied. The latter 
may eventually be achieved by applying dimension reduction techniques 
(see Section \ref{sens}).

Moreover, when recalling the results in \cite{WaSl08}, one may hope
that the convergence rate for the terms in (\ref{est3}) is even better.

Finally, we note that the constants involved in the estimate 
(\ref{errorrate}) may be chosen to be uniform with respect to $x\in X$.
Together with the perturbation estimates (\ref{valest}) and
(\ref{solest}) in Section \ref{twostage} one, hence, obtains
\begin{eqnarray*}
|v(P)-v(P_{n})|&\leq&\hat{C}(\delta)n^{-1+\delta}+O(\varepsilon),\\
S(P_{n})&\subseteq&S(P)+\psi_{P}^{-1}(\hat{C}(\delta)n^{-1+\delta}+O(\varepsilon))
\end{eqnarray*}
if $d_{S}(\varepsilon)\leq 2$. Here, $P_{n}$ is the discrete 
probability measure representing the QMC method, i.e., 
$P_{n}=n^{-1}\sum_{j=1}^{n}\delta_{\xi^{j}}$, where $\delta_{\xi}$ 
denotes the Dirac measure placing unit mass at $\xi$.
\end{remark}

\section{Orthogonal transformations and the Gaussian case}
\label{generic}

We consider the stochastic program (\ref{twost}) with
\[
\Phi(x,\xi)=\phi(h(\xi)-Tx)
\]
as in Section \ref{anovatwostage} and assume that (A1)--(A3) is
satisfied. Further we assume that $h(\xi)$ is of the form
$h(\xi)=(Q\xi,\bar{h})$ with some orthogonal $d\times d$ matrix $Q$
and with $\xi$ satisfying (A4). Then the relevant integrand is of
the form
\[
f(\xi)=\max_{j=1,\ldots,\ell}\langle v^{j},(Q\xi,\bar{h})-Tx\rangle
=\max_{j=1,\ldots,\ell}\langle\hat{Q}^{\top}v^{j},(\xi,\bar{h})-
\hat{Q}^{\top}Tx\rangle,
\]
where the $r\times r$ matrix $\hat{Q}$ is given by
\begin{equation}\label{hQ}
\hat{Q}=\left(\begin{array}{cc} Q\, & 0 \\ 0\, & I
\end{array}\right)
\end{equation}
with $I$ denoting the $(r-d)\times(r-d)$ identity matrix. Hence, the
results of Section \ref{anovatwostage} apply if the vertices
$\hat{Q}^{\top}v^{j}$, $j=1,\ldots,\ell$, of the linearly
transformed dual feasible set $\hat{Q}^{\top}\cD$ satisfy the
corresponding assumptions. The set $\hat{Q}^{\top}\cD$ may be
represented in the form
\[
\hat{Q}^{\top}\cD=\{\hat{Q}^{\top}z:W^{\top}z\le q\}=\{z\in\R^{r}:
(\hat{Q}^{\top}W)^{\top}z\le q\}.
\]
The geometric condition on the vertices is violated only if some
face of $\hat{Q}^{\top}\cD$ is parallel to some coordinate axis.
Clearly, there are only countably many orthogonal matrices $Q$ for
which this is the case.

Assume now that $\xi$ is normally distributed with zero mean and
nonsingular covariance matrix $\Sigma$. Let the nonsingular diagonal
matrix $D$ be the result of a unitary decomposition of $\Sigma$,
i.e., $D=Q\,\Sigma\,Q^{\top}$ with an orthogonal matrix $Q$. If
$h(\xi)=(\xi,\bar{h})$ enters the integrand (\ref{integr}) with
given dual feasible polyhedron $\cD$ and vertices $v^{j}$,
$j=1,\ldots,\ell$, and $\hat{Q}$ is defined as in (\ref{hQ}), the
integrand may be rewritten as
\[
f(\xi)=\max_{j=1,\ldots,\ell}\langle
\hat{Q}v^{j},(Q\xi,\bar{h})-\hat{Q}Tx\rangle.
\]
As $Q\xi$ is normal with covariance matrix $D$ and, thus, satisfies
(A4), the results of the preceding section apply when using the
transformed dual feasible set $\hat{Q}\cD$ and normal cones
$\hat{Q}^{\top}\cK_{j}$, $j=1,\ldots,\ell$, respectively. However,
given $\cD$, there are only countably many orthogonal matrices $Q$
such that the geometric condition on the vertices of $\hat{Q}\cD$ is
\underline{not} satisfied. When equipping the metric space of all
orthogonal $d\times d$ matrices with the standard norm topology, the
set of all orthogonal matrices $Q$ such that $\hat{Q}\cD$ satisfies
the algebraic condition on the vertices is {\em residual}, i.e., it 
may be represented as countable intersection of open dense subsets. 
It is said that a property is {\em generic} or holds for {\em almost 
all} elements of a metric space if it holds in a residual set.

\begin{corollary}\label{c2}
Let $x\in X$ and assume (A1)--(A3) with $h(\xi)=(\xi,\bar{h})$ with
fixed $\bar{h}\in\R^{r-d}$ to be satisfied.
\begin{itemize}
\item[(a)] The geometric condition that all components of all adjacent
vertices of $\hat{Q}\cD$ are distinct is a generic property in the
space of all $d\times d$ orthogonal matrices $Q$ where $\hat{Q}$ is
defined by (\ref{hQ}).
\item[(b)] Let $\xi$ be normally distributed with mean $m\in\R^{d}$
and nonsingular covariance matrix $\Sigma$, and let the orthogonal
matrix $Q$ be chosen such that $Q\,\Sigma\,Q^{\top}={\rm diag}
(\sigma_{1}^{2},\ldots,\sigma_{d}^{2})$. Let $\rho$ be the normal 
density with mean $m$ and covariance matrix ${\rm diag}(\sigma_{1}^{2},
\ldots,\sigma_{d}^{2})$. If $Q$ belongs to the residual set of orthogonal 
matrices satisfying the generic property, the ANOVA approximation $f^{(2)}$ 
of $f$ given by (\ref{secordanova}) belongs to the tensor product Sobolev 
space $\cW_{2,{\rm mix}}^{(1,\ldots,1)}(\R^{d})$.
\end{itemize}
\end{corollary}

\begin{proof}
While (a) is shown above, it remains to note for part (b) that (A4) is 
satisfied and, hence, the result follows from Theorem \ref{t3}.\hfill$\Box$
\end{proof}

\section{Sensitivity and dimension reduction of two-stage stochastic programs}
\label{sens}

In this section we discuss sensitivity and possibilities for
reducing the effective dimension of two-stage models. First, we
derive an upper bound for the global sensitivity indices
$\bar{S}_{\{i\}}$, $i=1,\ldots,d$, and the mean dimension
$\bar{d}_{S}$ in the superposition sense, respectively.
\begin{proposition}\label{p4}
Let (A1)--(A4) with $h(\xi)=(\xi,\bar{h})$ with fixed
$\bar{h}\in\R^{r-d}$ be satisfied and $\sigma_{i}^{2}$ denote the
variance of $\xi_{i}$, $i=1,\ldots,d$. Then 
\begin{eqnarray*}
\bar{S}_{\{i\}}&\le&\frac{\sigma_{i}^{2}}{\sigma^{2}(f)}\max_{j=1,\ldots,\ell}
|v_{i}^{j}|^{2}\quad(i=1,\ldots,d)\\
\bar{d}_{S}&\le&\frac{1}{\sigma^{2}(f)}\max_{j=1,\ldots,\ell}\|v^{j}\|_{\infty}^{2}
\sum_{i=1}^{d}\sigma_{i}^{2},
\end{eqnarray*}
where $v^{j}$, $j=1,\ldots,\ell$, are the vertices of the dual
polyhedron.
\end{proposition}
\begin{proof}
We use \cite[Theorem 3]{SoKu09} and compute the partial derivatives
of $f$ with respect to $\xi_{i}$, $i=1,\ldots,d$, which exist almost
everywhere on $\R^{d}$. If $h(\xi)-Tx$ belongs to the cone
$\cK_{j}$, then
\[
f(\xi)=\sum_{i=1}^{d}v_{i}^{j}(\xi_{i}-[Tx]_{i})+\sum_{i=d+1}^{r}v_{i}^{j}
(\bar{h}_{i}-[Tx]_{i}),
\]
where $x\in X$ is fixed. We obtain for $\xi\in\R^{d}$ such that
$h(\xi)-Tx$ belongs to the interior of $\cK_{j}$ that
\[
\frac{\partial f}{\partial \xi_{i}}=v_{i}^{j}.
\]
Hence, the partial derivative is piecewise constant and may be
bounded from above by $\max_{j=1,\ldots,\ell}|v_{i}^{j}|$. Using
\cite[Theorem 3]{SoKu09} this proves our estimate for the global
sensitivity index $\bar{S}_{\{i\}}$. The second estimate is a
consequence of formula (\ref{mdimest}).\hfill$\Box$
\end{proof}
Proposition \ref{p4} indicates that the importance of variable
$i$ on $f$ gets lower if $\sigma_{i}$ gets smaller.

If $\xi$ is normal with nonsingular covariance matrix $\Sigma$, the
{\em standard} (lower triangular) Cholesky matrix $L_{C}$ performing
the factorization $\Sigma=L_{C}L_{C}^{\top}$ seems to assign the same
importance to every variable and, hence, is not suitable to reduce the
effective dimension (at least in the truncation sense). This fact is
confirmed in our numerical experiments (see Section \ref{compexp}).

A universal principle for dimension reduction in the normal case is
{\em principal component analysis} (PCA). It is universal in the
sense that it does not depend on the structure of the underlying
integrand $f$. The basic idea of PCA is to determine the best mean 
square approximation of the form $\sum_{i=1}^{d}v_{i}z_{i}$ to a 
$d$-dimensional normal random vector $\xi$, where $v_{i}\in\R^{d}$,
$i=1,\ldots,d$, and $(z_{1},\ldots,z_{d})$ is normal with mean 
$0$ and covariance matrix $I$. The solution is 
$v_{i}=\sqrt{\lambda_{i}}u_{i}$ and 
$z_{i}=(\sqrt{\lambda_{i}})^{-1}u_{i}^{\top}\xi$, where
$\lambda_{1}\ge\cdots\ge\lambda_{d}>0$ are the eigenvalues of
$\Sigma$ in decreasing order and $u_{i}$, $i=1,\ldots,d$, the
corresponding orthonormal eigenvectors (see \cite{WaSl11}). Hence, 
PCA consists in using the factorization
\[
\Sigma=U_{P}\,U_{P}^{\top}\quad\mbox{or}\quad\Sigma=(u_{1},\ldots,u_{d})
{\rm diag}(\lambda_{1},\ldots,\lambda_{d})(u_{1},\ldots,u_{d})^{\top},
\]
where $U_{P}=(\sqrt{\lambda_{1}}u_{1},\ldots,\sqrt{\lambda_{d}}u_{d})$.
Several authors report an enormous reduction of the effective 
truncation dimension in financial models if PCA is used (see, for 
example, \cite{WaFa03,WaSl05,WaSl07}). We observed the same effect 
in our numerical experiments (see Section \ref{compexp}). However, 
the reduction effect certainly depends on the eigenvalues of 
$\Sigma$. If the ratio $\frac{\lambda_{1}}{\lambda_{d}}$ is close 
to $1$, the performance of PCA gets worse. Nevertheless we recommend 
to use first PCA and to resort to other ideas only after its failure.

Several other {\em dimension reduction techniques} exploit the fact 
that a normal random vector $\xi$ with mean $\mu$ and covariance matrix 
$\Sigma$ can be transformed by $\xi=B\eta+\mu$ and \underline{any} matrix 
$B$ satisfying $\Sigma=B\,B^{\top}$ into a standard normal random vector 
$\eta$ with independent components. The choice of $B$ may change the
QMC error and the effective dimension of the integrand $f_{x}$. 
As observed in \cite{Papa02,WaSl11}, however, there is no consistent 
dimension reduction effect for any such matrix $B$. This means that
a specific choice of the matrix $B$ may result in a dimension reduction 
for one integrand, but eventually not for another one. 

The following observation is seemingly due to \cite{Papa02}, too 
(see also \cite[Lemma 1]{WaSl11}).

\begin{proposition}\label{p5}
Let $\Sigma$ be a $d\times d$ nonsingular covariance matrix and $A$
be a fixed $d\times d$ matrix such that $A\,A^{\top}=\Sigma$. Then
$\Sigma=B\,B^{\top}$ if and only if $B$ is of the form
$B=A\,Q$ for some orthogonal $d\times d$ matrix $Q$.
\end{proposition}

To apply the proposition, one may choose $A=L_{C}$ since computing
the standard Cholesky matrix $L_{C}$ requires only $\frac{1}{6}d^{3}$ 
operations. Then any other decomposition matrix $B$ with 
$\Sigma=B\,B^{\top}$ is of the form $B=L_{C}\,Q$ with some 
orthogonal matrix $Q$. The approach proposed in \cite{ImTa04} for 
linear functions $f(\xi)=w^{\top}\xi + a$ consists in determining 
a {\em good} orthogonal matrix $Q$ by minimizing the mean 
truncation dimension (\ref{meandim}). This approach is extended 
in \cite{WaSl11} to functions $f$ of the form
$$
f(\xi)=G(w_{1}^{\top}\xi+a_{1},\ldots, w_{\ell}^{\top}\xi+a_{\ell})
$$
for some function $G$ and $w,\,w_{i}\in\R^{d}$, $a,\,a_{i}\in\R$, 
$i=1,\ldots,\ell$.
The latter is applicable to linear two-stage integrands if the function 
$G$ is chosen as $G(t_{1},\ldots,t_{\ell})=\max\{t_{1},\ldots,t_{\ell}\}$ 
and $w_{i}$ contains the first $d$ components of the vertex $v^{i}$ of the 
dual feasible set $\cD$ (see Proposition \ref{p4}). Of course, applying 
the orthogonalization techniques developed in \cite{WaSl11} to two-stage 
integrands is not straightforward since the vertices $v^{j}$ of $\cD$ are 
not known in general and the computation of all of them is too expensive.
So, its application to two-stage stochastic programs requires further
work.

For general (non-normal) random vectors $\xi$ the influence (of
groups) of variables and the computation of effective dimensions are
studied, e.g., in \cite{DrHo06,Sobo01,SoKu09,WaFa03}.

\section{Numerical experiments}
\label{compexp}

For our tests we consider a two-stage stochastic production planning problem 
which consists in minimizing costs of a company. The company aims to satisfy 
stochastic demands $\xi_{t}$ in a time horizon $\{1,\ldots,T\}$ with 
multivariate probability distribution $P$ (on $\R^{T}$), but its 
production capacity based on $I$ company owned units does eventually not
suffice to cover the demand. Hence, it has to buy the necessary amounts
from other $m=m_1+m_2$ providers or markets at fixed prices $\bar{c}_{1,j_1,t}$ 
and $\bar{c}_{2,j_2,t}$, $t=1,\ldots,T, 1 \le j_1 \le m_1, 1 \le j_2 \le m_2$, 
and aims at minimizing the expected costs.

The optimization model is of the form
\[
\min_{x\in\mathbb{R}^{IT}}\Big\{\sum_{t=1}^{T} \sum_{i=1}^{I} c_{i,t}\,x_{i,t}  +
\int_{\R^{T}}\Phi(x,\xi) P(d\xi)\,:\,x\in X \Big\},
\]
where the recourse costs $\Phi$ are given by
\[
\Phi(x,\xi)=\min_{y\in\R^{(m_1+m_2)T}}\Big\{\sum_{t=1}^{T}
\Big(\sum_{j_1=1}^{m_1}\bar{c}_{1,j_1,t}\,y_{j_1,t} + \sum_{j_2=1}^{m_2}\bar{c}_{2,j_2,t}\,y_{m_1+j_2,t} 
\Big):y\in Y(x,\xi)\Big\},
\]
with the polyhedral constraint sets
\[
X:=\left\{x\in\R^{IT}\left|
\begin{aligned}
& \,a_{i,t} \le x_{i,t} \le b_{i,t} \,,i=1,\dots,I , t=1,\dots,T \\
& \, | x_{i,t} - x_{i,t+1}| \le \delta_{i,t} \,,i=1,\dots,I ,t=1,\dots,T-1
\end{aligned}
\right\}\right.,
\]
and
\[
Y(x,\xi):=\left\{y\in\R^{mT}\left|
\begin{aligned}
& \sum_{i=1}^{I} x_{i,t} + \sum_{j=1}^{m_1+m_2} y_{j,t} \ge \xi_t \,,\\
& w_{1,j_1,t} \le y_{j_1,t} \le z_{1,j_1,t}\,,j_1=1,\dots,m_1\\
& w_{2,j_2,t} \le y_{m_1+j_2,t}\,,j_2=1,\dots,m_2\\
& (t=1,\dots,T) \\
& | y_{j_1,t} - y_{j_1,t+1}| \le \rho_{1,j_1,t}\, ,j_1=1,\dots,m_1,\\
& | y_{m_1+j_2,t} - y_{m_1+j_2,t+1}| \le \rho_{2,j_2,t}\,,j_2=1,\dots,m_2\\
& (t=1,\dots,T-1)
\end{aligned}
\right\}\right.
\]
with fixed positive prices $c_{i,t},\bar{c}_{1,j_1,t},\bar{c}_{2,j_2,t}$ and bounds 
$a_{i,t}$, $b_{i,t}$, $\delta_{i,t}$, $w_{1,j_1,t}$, $w_{2,j_2,t}$, $z_{1,j_1,t}$, 
$\rho_{1,j_1,t}$, $\rho_{2,j_2,t}$. 
We assume that the demands $\xi_t$ follow the condition
\begin{equation}\label{eq:ARMA_rel}
\xi_{t}= m_t + \eta_t, \quad \text{for } 1\le t \le T,
\end{equation}
where $m=(m_1,\dots,m_T)$ is a vector of expected values simulating the 
\emph{trend} or \emph{seasonality} of the demands, and $\eta$ is an 
\emph{ARMA(p,q) process} given by the recurrence equation
\begin{equation}\label{eq:ARMA_rec}
\eta_t=\sum_{i=1}^{p} \alpha_i \eta_{t-i} + \sum_{j=1}^{q}\beta_j \gamma_{t-j}
 + \gamma_{t}\quad(t\in\Z)
\end{equation}
with i.i.d. Gaussian noise $\gamma_t \sim$ N(0,1) and characteristic polynomials
$P(z)=1-\sum_{i=1}^{p}\alpha_{i}z^{i}$ and $Q(z)=1+\sum_{i=1}^{q}\beta_{i}z^{i}$.
An ARMA(p,q) process is stationary (i.e., the covariance function $R(t,s)=
\mathbb{E}(\eta_{t}\eta_{s})$ is of the form $R(t,s)=\lambda(|t-s|+1)$, $1\le t,s\le T$) 
iff the polynomials $P$ and $Q$ do not have common zeros and $P(z)\neq 0$ for 
all $z\in\mathbb{C}$ with $|z|\leq 1$ (see \cite[Chapter 3]{BrDa02}).

The vector of demands $\xi_1,\dots,\xi_T$ is then normally distributed with 
mean vector $m$ and covariance matrix dependending on the constants 
$\alpha_i$, $\beta_j$, $1 \le i \le p$, $1 \le j \le q$, $p,q\in\N$.
Such models have been considered for simulating electricity load demands in 
energy industry, see e.g. \cite{PE+10} and \cite{EiRW05}. Note that since 
the model includes unbounded demands $\xi$, no upper bounds in the variables $y_{m_1+j_2,t}$, $j_2=1,\dots,m_2$, $t=1,\dots,T$, were imposed, allowing to 
cover arbitrarily large demand values. We select in addition the prices 
$\bar{c}_{2,j_2,t}$ significantly higher than the prices $\bar{c}_{1,j_1,t}$, 
such that the variables $y_{m_1+j_2,t}$, $j_2=1,\dots,m_2$, $t=1,\dots,T$, 
do not represent always the trivial choice for costs minimization. 
For our tests, we choose the real dimension $d=T=100$, and the model 
constants $p=2$, $q=6$, $\alpha_1=-0.52$, $\alpha_2=0.45$, $\beta_1=-0.17$, 
$\beta_2=0.12$, $\beta_3=0.05$, $\beta_4=-0.07$, $\beta_5=0.06$, 
$\beta_6=0.04$. The resulting ARMA process $\eta$ is stationary and, hence, 
$\mathbb{E}(\eta_{t}\eta_{s})=\lambda(|t-s|+1)$, $1\le t,s\le T$. The values 
$\lambda(t), 1\le t \le T$, can be obtained by solving a system of linear 
equations with coefficients depending on the constants $\alpha_i$, 
$\beta_j$, $1 \le i \le 2$, $1 \le j \le 6$ (see \cite{BrDa02} for 
detailed information about modeling with ARMA processes). The resulting 
covariance matrix $\Sigma$ is Toeplitz symmetric, with entry values 
$\Sigma(i,j)=\lambda(|i-j|+1)$. The integration problem is transformed by factorizing the covariance matrix $\Sigma=A\,A^{\top}$ as usually 
recommended in Gaussian high-dimensional integration (see 
\cite[Sect. 2.3.3]{Glas04}). We carry out our tests using the Cholesky factorization $A=L_{C}$ (CH) and the principal component analysis 
factorization $A=U_{P}$ (PCA) (see Section \ref{sens}). After the 
factorization of $\Sigma$ assumptions (A1)--(A4) (see Section 
\ref{twostage}) are satisfied. Hence, Theorem \ref{t3} applies if 
(A5) is satisfied. 

A simulated demands-path \mbox{$\xi_{1},\dots,\xi_{d}$} can then
be obtained by
\[
(\xi_{1},\dots,\xi_{d})^{\top}=A\,(\phi^{-1}(z_1),\dots,\phi^{-1}(z_d))^{\top} 
+ (m_1,\dots,m_d),
 \]
where $Z=(z_{1},\ldots,z_{d})\sim U([0,1]^d)$ (i.e., the probability
distribution of $Z$ is uniform on $[0,1]^{d}$), and $\phi^{-1}(.)$ 
represents the inverse cumulative normal distribution function, which 
can be efficiently and accurately calculated by Moro's algorithm (see 
\cite[Sect. 2.3.2]{Glas04}). The evaluation begins then with MC or 
randomized QMC points for the samples $Z \sim U([0,1]^d)$. For MC
points in $[0,1]^d$ we used the Mersenne Twister \cite{MaNi98} as
pseudo random number generator. For QMC, we use randomly scrambled
Sobol' points with direction numbers given in \cite{JoKu03} and
randomly shifted lattice rules \cite{SlKJ02,KuSS11}. The implemented
scrambling technique is random linear scrambling described in
\cite{Mato98}. For our tests, we considered cubic decaying weights
$\gamma_{j}=\frac{1}{j^{3}}$ for constructing the lattice rules.

We chose the following parameters for the numerical experiments:
\begin{itemize}
\item $I=10$, $m_1=6$, $m_2=2$.  
\item For all $i,j_1,j_2,t,$ we select randomly $a_{i,t} \in [0.001,0.003]$, $b_{i,t}\in [0.3,0.6]$, $\delta_{i,t}\in [0.3,0.35]$, $w_{1,j_1,t}, 
w_{2,j_2,t}\in[0.000001,0.00002]$, $z_{1,j_1,t}\in [5,7]$, and 
$\rho_{1,j_1,t},\rho_{2,j_2,t} \in [1.0,1.1]$.
\item For all $i,j_1,j_2,t,$ we select randomly $c_{i,t} \in [7,9]$, 
$\bar{c}_{1,j_1,t}\in [8,10]$, and $\bar{c}_{2,j_2,t} \in[12,14]$.
\end{itemize}
The given parameters were chosen to attempt avoiding trivial solutions 
of the linear programs. \\
We perform two different kind of tests in our experiments. For the first 
kind of tests we fix $n$ sampling points $\xi^{j}$ and replace the integral 
of the second stage function $\Phi(x,\cdot)$ by the equal weight MC and 
randomized QMC quadrature rule, respectively.
Then we solve the resulting large linear program
\begin{equation}\label{eq:Tests_Firstkind}
\min_{x\in\mathbb{R}^{IT}}\Big\{\sum_{t=1}^{T} \sum_{i=1}^{I} c_{i,t}\,x_{i,t} +
\frac{1}{n}\sum_{j=1}^{n}\Phi(x,\xi^{j}) P(d\xi)\,:\,x\in X \Big\}.
\end{equation}
For the second kind of tests, we select fixed feasible points $\,x\in X$ 
and examine the integration errors for the expected recourse
\begin{equation}\label{eq:Tests_Secondkind}
\int_{\R^{T}}\Phi(x,\xi) P(d\xi)
\end{equation}
by equal weight MC or randomized QMC quadrature rules. 
\begin{figure}[h]
\hspace{-1cm}
\centering
\includegraphics[width=12cm, height=7cm]{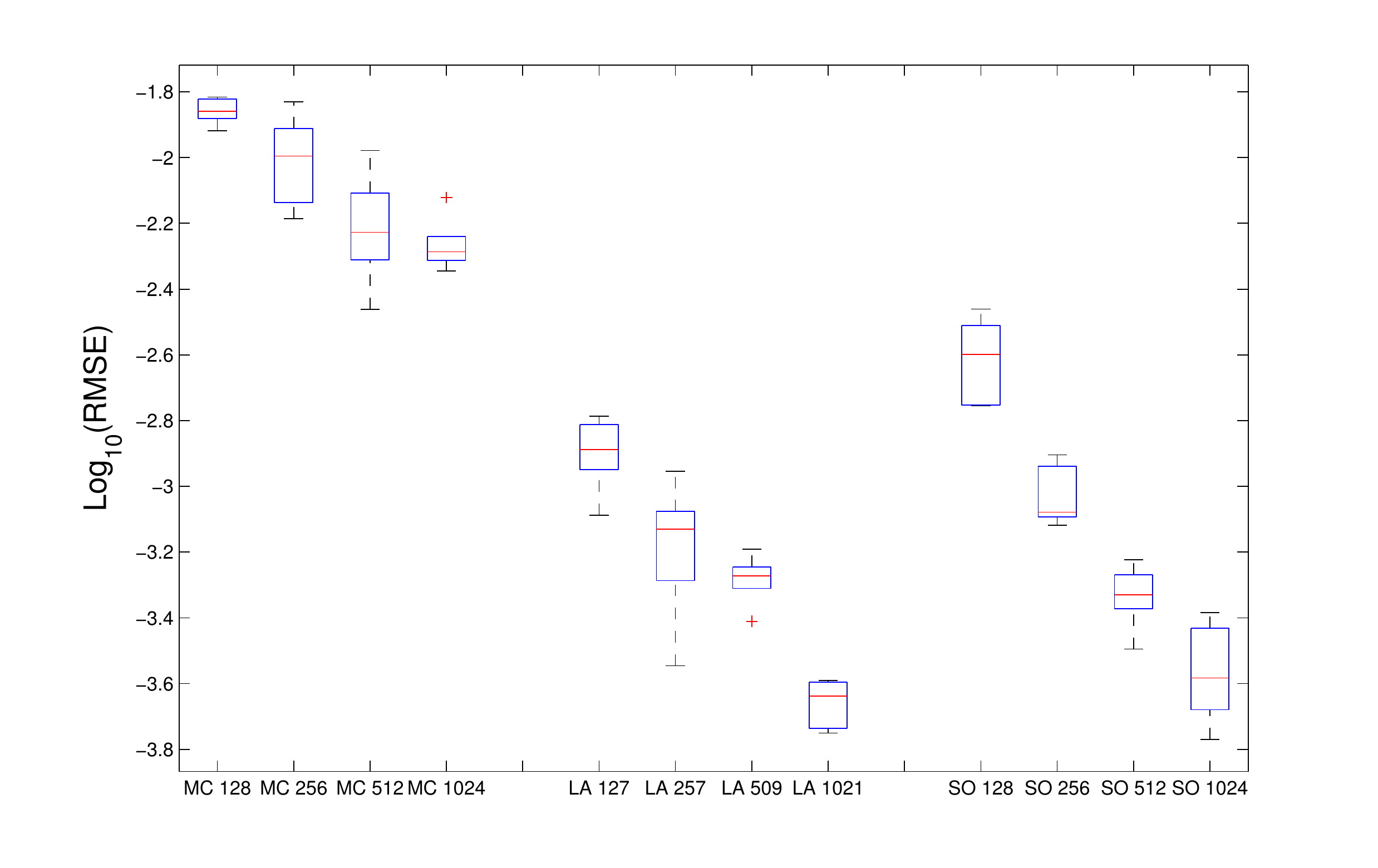}\\
\vspace{-0,5cm}\hspace{-0,9cm}
\includegraphics[width=12cm, height=7cm]{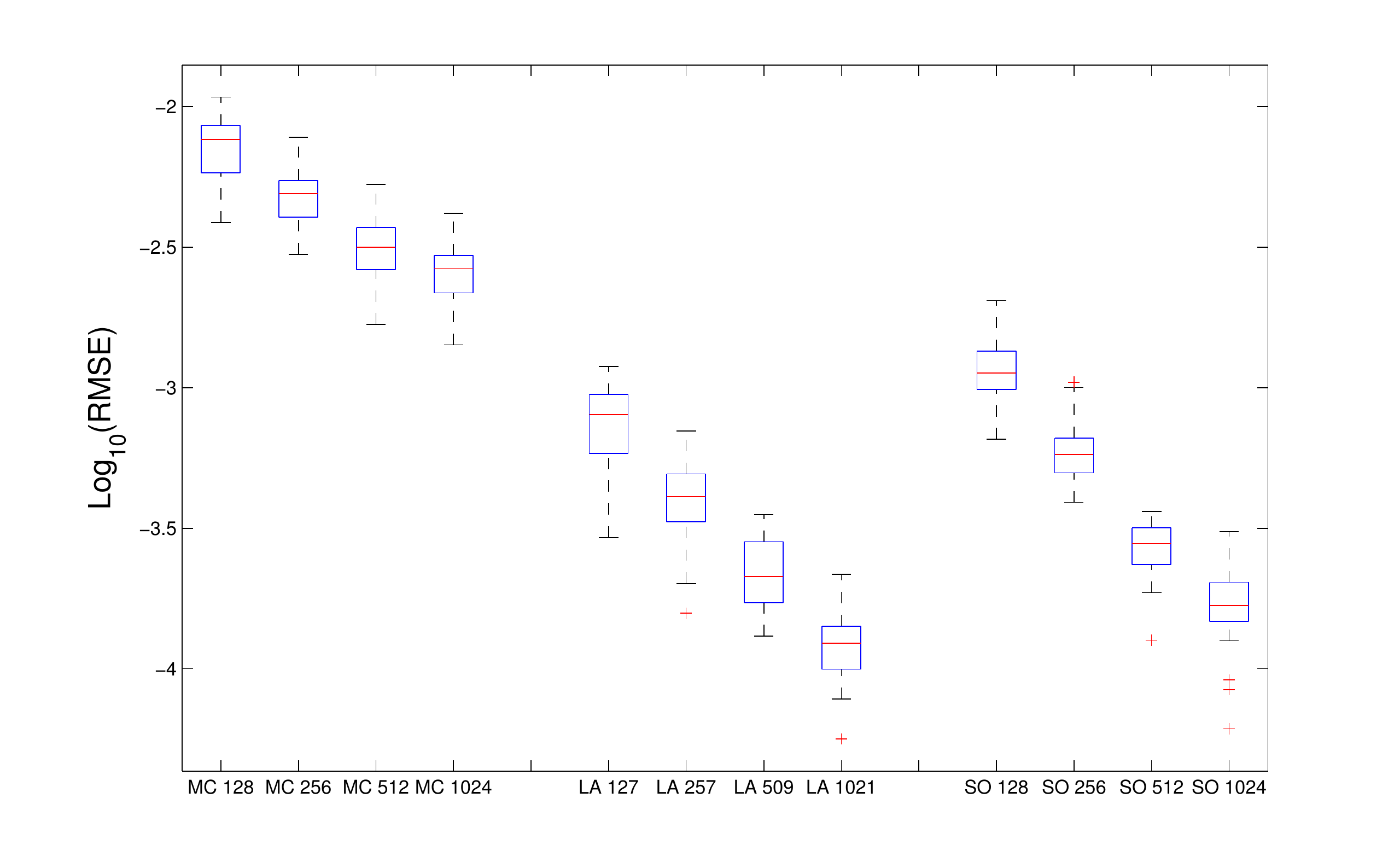}
\caption{Shown are the  $Log_{10}$ of relative RMSE with PCA factorization of
      covariance matrix for integrating $\Phi(x,\cdot)$ (upper figure) and for
      the minimum in \eqref{eq:Tests_Firstkind} (lower figure). Results for 
      Mersenne Twister MC and randomly scrambled Sobol' QMC with $128,256,512$ 
      and $1024$ points (MC $128$,... or SO $128$,...), and randomly shifted 
      lattice rules QMC with $127,257,509$ and $1021$ lattice points (LA $127$,...).
     }    
\label{fig:PCA}
\end{figure}
For simplicity 
we choose the fixed feasible points $x\in X $ to be the optimal solutions 
of the tests of the first kind, which were obtained by solving the resulting 
linear program for different costs while keeping the constraint set unchanged. 
The aim of these experiments is twofold. First we examine the convergence rate 
of the MC or randomized QMC quadrature rules with some fixed feasible points 
$\,x\in X$ for the expected recourse in the tests of second kind. Secondly we 
examine if these convergence rates in terms of sample sizes $n$ are translated 
to the resulting large linear programs for the tests of first kind.
\begin{figure}[h]
\hspace{-1cm}
\centering
\includegraphics[width=12cm, height=7cm]{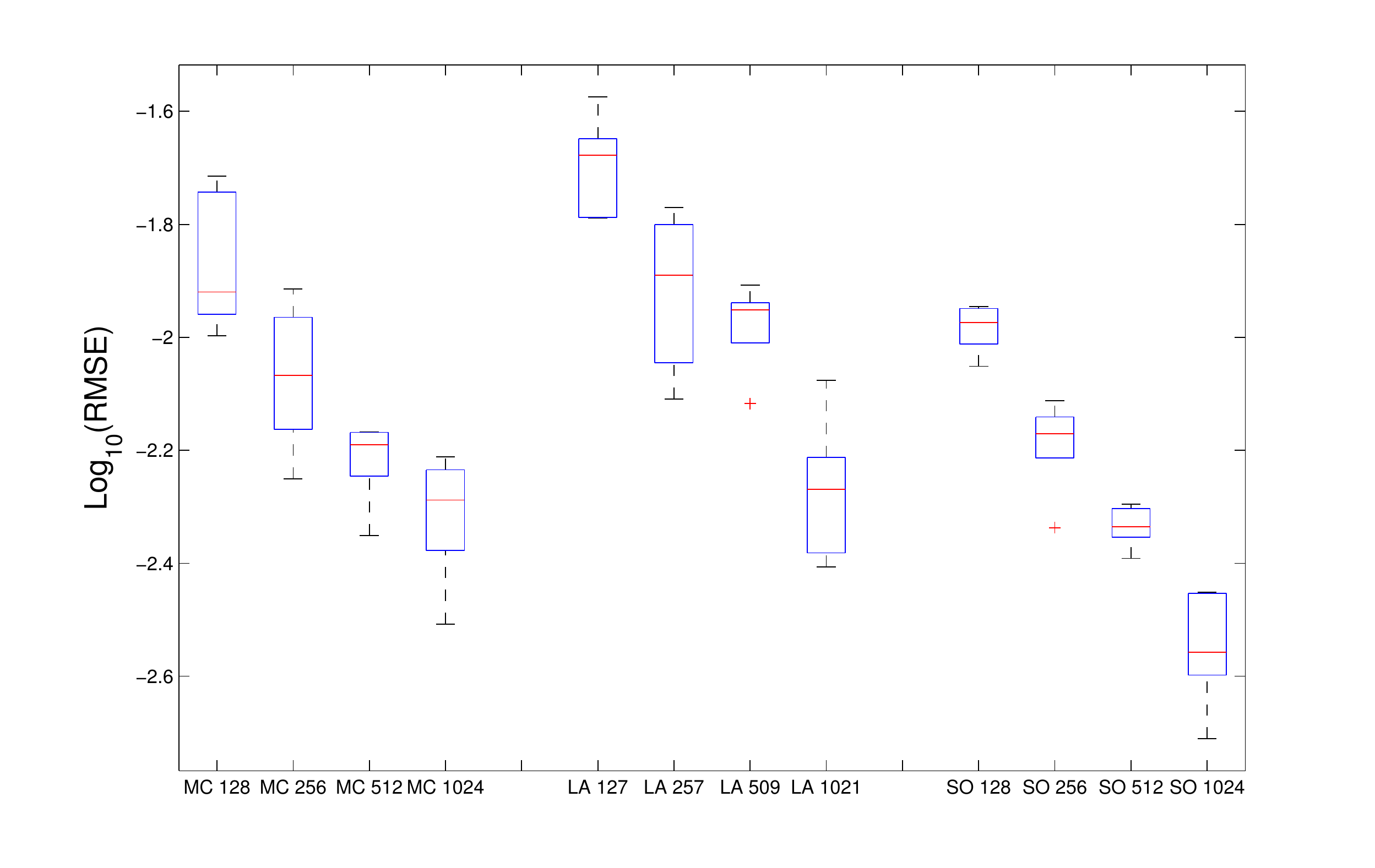}\\
\vspace{-0,5cm}\hspace{-0,9cm}
\includegraphics[width=12cm, height=7cm]{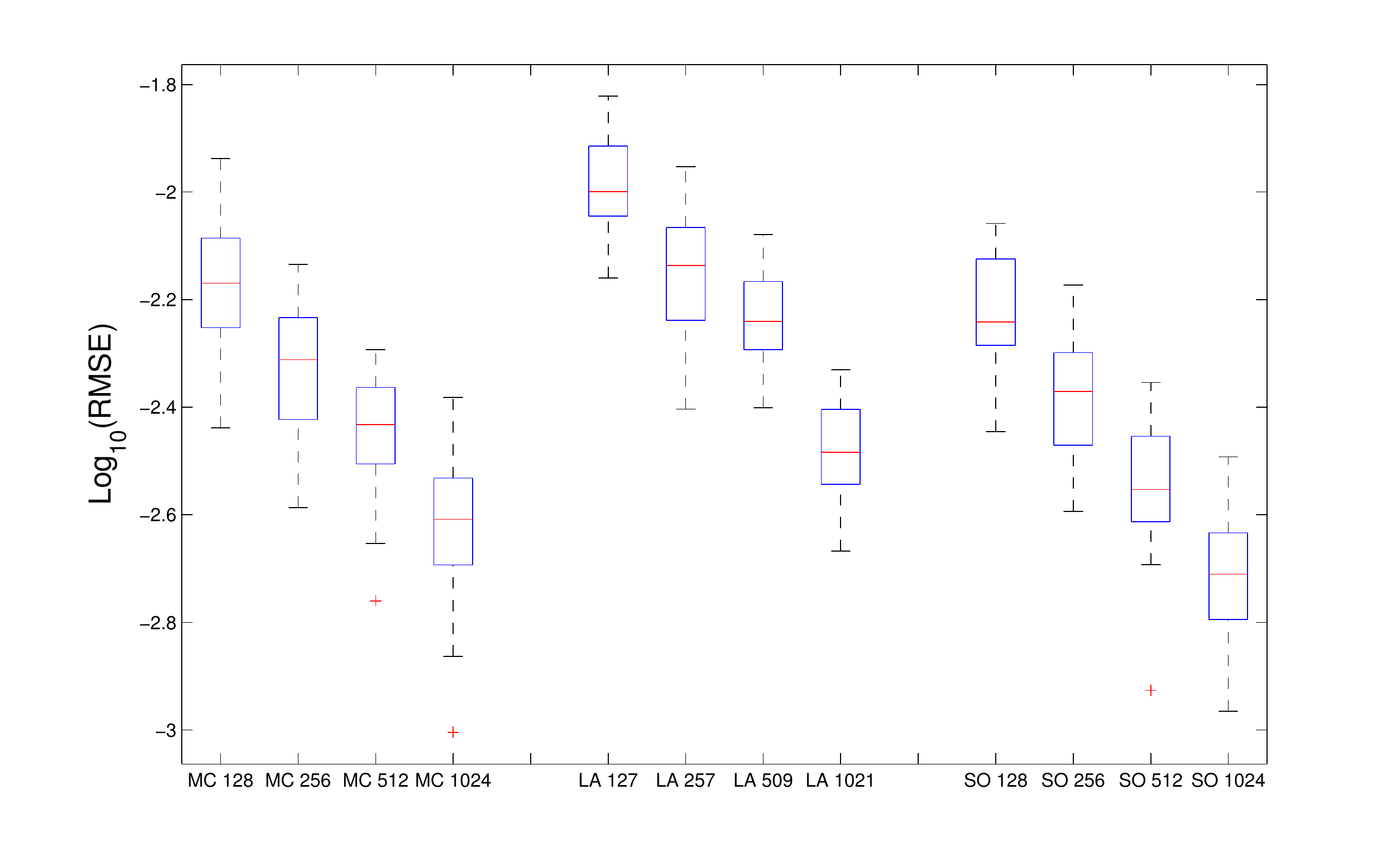}
\caption{Shown are the  $Log_{10}$ of relative RMSE with Cholesky factorization 
      of covariance matrix for integration of $\Phi(x,\xi)$ (upper figure) and 
      for the minimum in \eqref{eq:Tests_Firstkind} (lower figure). Results for 
      Mersenne Twister MC and randomly scrambled Sobol' QMC with $128,256,512$ 
      and $1024$ points (MC $128$,... or SO $128$,...), and randomly shifted 
      lattice rules QMC with $127,257,509$ and $1021$ lattice points (LA $127$,...).
     }
\label{fig:CH}
\end{figure}
The results for the tests of first and second kind under PCA factorization 
are summarized in Figure \ref{fig:PCA}. We chose $n=128,256,512,1024$ for 
the Mersenne Twister and for Sobol' points. For randomly shifted lattices, 
we chose the primes $n=127,257,509,1021$. The random shifts were generated 
using the Mersenne Twister. We estimate the relative root mean square errors 
(RMSE) of the estimated integrals (for the tests of the first kind) and of 
the optimal objective values (for the tests of the second kind) by taking 
$10$ runs of every experiment, and repeat the process $30$ times for the box 
plots in the figures. The box-plots show the first (lower bound of the box)
and third quartiles (upper bound of the box), and the median (line between
lower and upper bound). Outliers are marked by plus signs and the remaining 
results lie between the bounds.

The average of the estimated rates of convergence for both kind of tests under 
PCA ranged in $[-0.95,-0.85]$ for randomly shifted lattice rules, and in 
$[-1,-0.9]$ for randomly scrambled Sobol' points, for different price- and 
bound-parameters as listed above. This is clearly superior to the MC convergence 
rate of $-0.5$. The effective truncation dimension of $\Phi(x,\cdot)$ was 
tested at $20$ different feasible vertices $x$ (obtained from the tests of 
first kind with different price parameters and fixed bounds). We used the
algorithm proposed in \cite{WaFa03} with $2^{16}$ randomly scrambled Sobol' 
points ensuring that all results for the ANOVA total and partial variances 
were obtained with at least $3$ digits accuracy. The effective dimension 
$d_{T}$ remained close to $2$ in most cases and always $\le 6$. 
Further tests for the case $d_{T}=6$ showed that the variance accumulated by the first order ANOVA 
terms $f_{\{i\}}, \; 1\le i \le 6$ was approximately $95 \%$ 
of the total variance. The first order ANOVA 
terms $f_{\{i\}}, \; 7\le i \le d$ accumulated in total approximately $0,5 \%$ 
of the total variance. Moreover, adding the variance of the ANOVA terms 
$f_{\{1,2\}}$ and $f^2_{\{1,3\}}$ to the variance of the terms 
$f_{\{i\}}, \; 1\le i \le 6$ resulted in a variance accumulation higher than $99 \%$. 
Therefore we can conclude that the effective superposition dimension for the PCA case
is $d_S(0.01)=2$ in this case. Intensive computations seem to show that we may have $d_S=2$ for even smaller 
values of $\epsilon$ than $0.01$.  
Hence, PCA serves as excellent dimension reduction technique.

Although the geometric condition (A5) seems difficult to prove in this case 
(and maybe in many high-dimensional realistic examples encountered in energy industry), 
we may rely on Corollary \ref{c2} which states that the condition
is satisfied for almost all covariance matrices except for countably many. 
Indeed, it seems that the recourse function $\Phi(x,\cdot)$ is well approximated 
by a low dimensional smooth function as is the case in many practical examples 
considered in finance (see \cite{GrKS13}), for different feasible vertices 
$x\in X$. Further tests were carried out by combining randomly shifted lattice rules 
with the tent transformation as described in \cite{Hick02}, but no improvements in the 
convergence rates beyond $O(n^{-1})$ were observed for our feasible range of sample sizes. 
Similarly no improvement beyond the rate 
$O(n^{-1})$ was observed for scrambled Sobol' sequences as might be expected
for smooth integrand (see Section \ref{rqmc})). This may be explained by the
lack of the required smoothness properties of the second order ANOVA 
approximation. 

Using the Cholesky factorization, the results for both kind of tests
were completely different than those under PCA. The average of the 
estimated rates of convergence of randomized QMC ranged in $[-0.6,-0.5]$,
which is very close to the expected MC rate of $-0.5$. The results for the
Cholesky factorization are presented in Figure \ref{fig:CH}. The effective 
truncation dimension of $\Phi(x,\cdot)$ was estimated to be equal to 
$d_{T}=100$, which is just the real dimension $d$ of the problem.
Tests showed that the variance accumulated by the first order ANOVA 
terms $f_{\{i\}}, \; 1\le i \le d$ was approximately $20 \%$ 
of the total variance. It seems very likely that the 
the effective superposition dimension for the Cholesky case is really high-dimensional.

\section{Conclusions}

Our theoretical results in Section \ref{anovatwostage} imply that
all ANOVA terms except the one of highest order of integrands $f$
appearing in linear two-stage stochastic programs are smoother
than $f$. More precisely, the ANOVA terms of first and second order
belong to the tensor product Sobolev space which is important for
optimal convergence rates of randomly shifted lattice rules. Error
estimates as in Remark \ref{r2} then indicate that we may expect 
that Quasi-Monte Carlo approximations of two-stage stochastic
programs converge with the optimal rate (\ref{rate}) even for high
dimensions $d$ if the effective superposition dimension satisfies
$d_{S}\le 2$. Since we estimate the effective truncation dimension
$d_{T}$ and it holds $d_{S}\le d_{T}$, it is important that $d_{T}$ 
is equal to or at least close to $2$. This requires the use of 
dimension reduction techniques, for example, principal component 
analysis for (log)normal probability distributions $P$. 

Our preliminary computational experience on applying Quasi-Monte
Carlo methods to a two-stage stochastic production planning problem
confirms the theoretical results. They show that using appropriate
Quasi-Monte Carlo methods instead of Monte Carlo may lead to a
substantial improvement, because one may work with a much smaller
number of scenarios if suitable dimension reduction techniques allow 
for an essential reduction from $d_{T}=d$ to $d_{T}$ close to $2$.

Altogether, there are good reasons to conclude that recent
Quasi-Monte Carlo methods (like (scrambled) Sobol' sequences and
randomly shifted lattice rules) {\em may be efficient} for two-stage 
linear stochastic programs (even if the programs are large scale) if 
they allow for a clear dimension reduction. However, our present
theoretical results do not support the use of higher order QMC methods
(see \cite{Dick08,DiPi10}) since the first and second order ANOVA terms
do not satisfy the required smoothness conditions.

\begin{acknowledgement}
The authors wish to express their gratitude to Prof. Ian Sloan
(University of New South Wales, Sydney) for inspiring conversations
during his visit of the Humboldt-University Berlin in 2011. The
research of the first author is partially supported  by a grant of
Kisters AG, the second by a grant of the German Bundesministerium
f\"ur Wirtschaft und Technologie (BMWi) and the third by the DFG
Research Center {\sc Matheon} at Berlin. The authors extend their
gratitude to two anonymus referees and to the Associate Editor 
for their constructive and stimulating criticism.
\end{acknowledgement}

\end{document}